\documentclass[reqno,11pt]{amsart}

\usepackage[a4paper,left=23mm,right=23mm,top=30mm,bottom=30mm,marginpar=25mm]{geometry}
\usepackage{amsmath}
\usepackage{amssymb}
\usepackage{amsthm}
\usepackage{amscd}
\usepackage{mathtools}
\usepackage{url}
\usepackage[ansinew]{inputenc}
\usepackage{color}
\usepackage[final]{graphicx}
\usepackage{esint} 
\usepackage{bbm}
\usepackage{subfigure}
\usepackage{enumitem}
\usepackage{tikz}
\usepackage[noadjust]{cite}
\usepackage[hidelinks]{hyperref}

\renewcommand{\epsilon}{\varepsilon}

\numberwithin{equation}{section}

\newtheoremstyle{thmlemcorr}{10pt}{10pt}{\itshape}{}{\bfseries}{.}{10pt}{{\thmname{#1}\thmnumber{ #2}\thmnote{ (#3)}}}
\newtheoremstyle{thmlemcorr*}{10pt}{10pt}{\itshape}{}{\bfseries}{.}\newline{{\thmname{#1}\thmnumber{ #2}\thmnote{ (#3)}}}
\newtheoremstyle{defi}{10pt}{10pt}{\itshape}{}{\bfseries}{.}{10pt}{{\thmname{#1}\thmnumber{ #2}\thmnote{ (#3)}}}
\newtheoremstyle{remexample}{10pt}{10pt}{}{}{\bfseries}{.}{10pt}{{\thmname{#1}\thmnumber{ #2}\thmnote{ (#3)}}}
\newtheoremstyle{ass}{10pt}{10pt}{}{}{\bfseries}{.}{10pt}{{\thmname{#1}\thmnumber{ A#2}\thmnote{ (#3)}}}

\theoremstyle{thmlemcorr}
\newtheorem{theorem}{Theorem}
\numberwithin{theorem}{section}
\newtheorem{lemma}[theorem]{Lemma}
\newtheorem{corollary}[theorem]{Corollary}
\newtheorem{proposition}[theorem]{Proposition}

\theoremstyle{thmlemcorr*}
\newtheorem{theorem*}{Theorem}
\newtheorem{lemma*}[theorem]{Lemma}
\newtheorem{corollary*}[theorem]{Corollary}
\newtheorem{proposition*}[theorem]{Proposition}
\newtheorem{problem*}[theorem]{Problem}
\newtheorem{conjecture*}[theorem]{Conjecture}

\theoremstyle{defi}
\newtheorem{definition}[theorem]{Definition}

\theoremstyle{remexample}
\newtheorem{remarkx}[theorem]{Remark}
\newenvironment{remark}
  {\pushQED{\qed}\remarkx}
  {\popQED\endremarkx}
\newtheorem{examplex}[theorem]{Example}
\newenvironment{example}
  {\pushQED{\qed}\examplex}
  {\popQED\endexamplex}

\theoremstyle{ass}

\newcommand{\Acal}{\mathcal{A}}

\newcommand{\Fcal}{\mathcal{F}}

\newcommand{\Pcal}{\mathcal{P}}

\newcommand{\Qcal}{\mathcal{Q}}

\newcommand{\Scal}{\mathcal{S}}

\DeclareMathOperator{\dist}{dist}

\DeclareMathOperator{\supp}{supp}

\newcommand{\normb}[1]{\bigl\|#1\bigr\|}

\newcommand{\abslr}[1]{\left|#1\right|}

\newcommand{\absb}[1]{\bigl|#1\bigr|}
\newcommand{\absB}[1]{\Bigl|#1\Bigr|}

\newcommand{\N}{\mathbb{N}}
\newcommand{\R}{\mathbb{R}}

\newcommand{\loc}{\mathrm{loc}}

\newcommand{\eps}{\epsilon}

\DeclareMathOperator{\Div}{div}

\def\XXint#1#2#3{{\setbox0=\hbox{$#1{#2#3}{\int}$}
\vcenter{\hbox{$#2#3$}}\kern-.5\wd0}}

\DeclarePairedDelimiter\abs{\lvert}{\rvert}
\DeclarePairedDelimiter{\norm}{\lVert}{\rVert}

\newcommand{\Rn}{\R^{n}}

\renewcommand{\phi}{\varphi}

\newcommand{\weakto}{\rightharpoonup}

\def\XXint#1#2#3{{\setbox0=\hbox{$#1{#2#3}{\int}$}
     \vcenter{\hbox{$#2#3$}}\kern-.5\wd0}}

\newcommand{\Rmn}{\mathbb{R}^{m \times n}}

\newcommand{\Lip}{\mathrm{Lip}}
\makeatletter
\g@addto@macro\bfseries{\boldmath}
\makeatother

\usepackage{esint}

\usepackage{appendix}

\DeclareMathOperator{\diver}{div}

\renewcommand{\O}{\Omega}

\renewcommand{\d}{\delta}

\newcommand{\f}{\varphi}

\usepackage{amsmath}
\usepackage{latexsym}
\usepackage{amssymb}

\def\XXint#1#2#3{{\setbox0=\hbox{$#1{#2#3}{\int}$}
		\vcenter{\hbox{$#2#3$}}\kern-.5\wd0}}

\usepackage{graphicx}
\usepackage{wrapfig}

\newcommand{\abss}[1]{\left\langle #1 \right\rangle}

\title[$\Gamma$-Convergence involving nonlocal gradients with varying horizon]
{$\Gamma$-Convergence involving nonlocal gradients with varying horizon: Recovery of local and fractional models}

\author{Javier Cueto}
\address{Department of Mathematics, Faculty of Sciences, Universidad Aut\'onoma de Madrid, 28049 Madrid, Spain}
\email{javier.cueto@uam.es}

\author{Carolin Kreisbeck}
\address{Mathematisch-Geographische Fakult\"at, Katholische Universit\"at Eichst\"att-Ingolstadt, Osten\-stra{\ss}e 28, 85072 Eichst\"att, Germany}
\email{carolin.kreisbeck@ku.de}

\author{Hidde Sch\"{o}nberger}
\address{Mathematisch-Geographische Fakult\"at, Katholische Universit\"at Eichst\"att-Ingolstadt, Osten\-stra{\ss}e 28, 85072 Eichst\"att, Germany}
\email{hidde.schoenberger@ku.de}

\begin{document}
\maketitle

\thispagestyle{empty}
\begin{abstract}  
This work revolves around the rigorous asymptotic analysis of models in nonlocal hyperelasticity. The corresponding variational problems involve integral functionals depending on nonlocal gradients with a finite interaction range $\delta$, called the horizon. After an isotropic scaling of the associated kernel functions, we prove convergence results in the two critical limit regimes of vanishing and diverging horizon. While the nonlocal gradients localize to the classical gradient as $\delta\to 0$, we recover the Riesz fractional gradient as $\delta\to \infty$, irrespective of the nonlocal gradient we started with. Besides rigorous convergence statements for the nonlocal gradients, our analysis in both cases requires compact embeddings uniformly in $\d$ as a crucial ingredient. These tools enable us to derive the $\Gamma$-convergence of quasiconvex integral functionals with varying horizon to their local and fractional counterparts, respectively.

\vspace{8pt}

\vspace{8pt}

\vspace{8pt}

\noindent\textsc{Date:} \today.
\end{abstract}

\maketitle
\thispagestyle{empty}

\section[Introduction]{Introduction}

Nonlocal-to-local limits constitute a pivotal aspect in nonlocal modeling. 
They can provide a useful consistency check in confirming the compatibility of a (new) nonlocal model with a local counterpart that is covered by well-established theories.  
More generally speaking, the asymptotic analysis of critical parameter regimes can yield new insights about limit models from their approximations and vice versa. In this paper, we address these topics in the context of models with nonlocal gradients, which have seen a great rise in interest in the last few years, see e.g.,~\cite{ShS2015, Comi1, Sil20, MeS, BeCuMC, BeCuMC23, KreisbeckSchonberger, EGM22}.

For a general radial kernel $\rho$, the nonlocal gradient of a function $u:\R^n\to \R^m$ associated to $\rho$ is defined as
\begin{align}\label{Drho_intro}
	D_{\rho} u(x)=\int_{\Rn} \frac{u(x)-u(y)}{|x-y|}\otimes \frac{x-y}{|x-y|} \rho(x-y) \, dy \quad \text{for $x \in \Rn$},
\end{align}
whenever this integral exists. A widely studied special case is the Riesz fractional gradient given by $D^s:=D_{\rho^s}$ with $\rho^s=|\cdot|^{-(n+s-1)}$ for the fractional parameter $s\in (0,1)$. It satisfies natural physical invariance requirements \cite{Sil20} and was brought to the spotlight by Shieh \& Spector \cite{ShS2015,ShS2018}, who established useful counterparts of results from classical Sobolev space theory and showed that the associated function spaces coincide with the Bessel potential spaces $H^{s,p}(\Rn;\R^m)$.  This paved the way for the forthcoming works on fractional variational problems \cite{BeCuMC,BeCuMC21,KreisbeckSchonberger}, which were proposed as alternatives to the standard models in continuum mechanics. Due to the reduced regularity requirements imposed by the fractional derivatives in comparison to the classical ones,  these models admit a broader class of admissible deformations, which allows to account also for discontinuity effects, such as fracture and cavitation.

Despite the desirable properties of the Riesz fractional gradient, an intrinsic drawback from the perspective of continuum mechanical modeling is that it involves integration over the whole space $\Rn$, which is not suitable for models on bounded domains. This shortcoming can be resolved by considering kernel functions $\rho$ with compact support, meaning that the range of interactions between individual points, called the horizon $\d>0$, is finite. In~\cite{BeCuMC23}, Bellido, Cueto \& Mora-Corral used finite-horizon fractional gradients as a basis to propose models of nonlocal hyperelasticity. 

Note that the concept of a horizon stems originally from peridynamics \cite{Silling2000, SiLe10}, a nonlocal formulation of continuum mechanics  that, in contrast to classical modeling, avoids the use of derivatives, and instead considers the interaction between individual particles that are not necessarily at an infinitesimal distance.
Since its introduction in the 2000s, it has led to a vast literature, ranging from applied to theoretical contributions, see e.g.~\cite{MaO14, BFG17, DDG19}. While energetic approaches in the context of bond-based peridynamics typically involve double-integrals, the energy functionals of nonlocal hyperelasticity, which are integrals depending on nonlocal gradients, can be interpreted in the context of state-based peridynamics~\cite{SEW07}; one of the advantages of this framework is that it allows to model a broad range of material properties, e.g., general Poisson ratios in isotropic elastic materials, in contrast to the bond-based formulation~\cite{SEW07}.

A first step in putting the nonlocal hyperelastic models on a solid mathematical foundation is to guarantee the existence of solutions, which has been addressed in~\cite{BeCuMC23, BeCuMC23b, CuKrSc23} for the case of finite-horizon gradients. The arguments rely essentially on two key techniques: 
a nonlocal version of the fundamental theorem of calculus~\cite{BeCuMC23}, which is needed to prove Poincar\'{e} inequalities and compact embeddings, and a translation method established in~\cite{BeCuMC23, CuKrSc23} that expresses the nonlocal gradient as the classical gradient composed with a convolution. These tools were recently extended in \cite{BMS24} to general nonlocal gradients with compactly supported radial kernels, and we utilize this in Section~\ref{sec:existence} to develop an existence theory in this broadened setting. \medskip

Our focus in this work lies on the study of the critical parameter regimes for the horizon $\delta$, which can be seen as an important next step towards understanding and validating nonlocal hyperelasticity. While the widely used bond-based models are only able to recover a considerably restrictive class of models through a nonlocal-to-local limit passage \cite{BeCuMC20, MeS23}, we establish in this paper that the models involving nonlocal gradients are compatible with their local counterpart via a vanishing horizon limit. For a complete picture of the horizon-dependence, we also analyze the other extreme regime of diverging horizon, providing a rigorous connection with purely fractional models.

In the following, we adopt the framework for general nonlocal gradients from the recent paper~\cite{BMS24} by Bellido, Mora-Corral \& Sch\"onberger. Starting with a fixed radial kernel $\rho$ satisfying the hypotheses of~\cite{BMS24} (see \ref{itm:h0}-\ref{itm:upper} in Section~\ref{sec:nonlocalgradients}) with horizon equal to 1, i.e., $\supp \rho = \overline{B_1(0)}$,
we apply an isotropic rescaling to obtain the kernels
\begin{equation}\label{eq:scalingintro}
\rho_\d  = c_\d \rho\left( \frac{\, \cdot\,}{\d}\right)
\end{equation}
with horizon $\d>0$ and scaling constants $c_\delta>0$ to be chosen suitably depending on the targeted parameter regime for $\d$,
cf.~(i) and (ii) below;
examples of admissible realizations of kernel functions $\rho$ can be found in Example~\ref{ex:kernels}.

Our aim in this work is to study the asymptotic behavior of the nonlocal gradients associated to the kernels $\rho_\d$ in~\eqref{eq:scalingintro} for both limits of vanishing and diverging horizon, that is,
\begin{center}
(i)\, $\d \to 0$ \qquad \text{and} \qquad (ii)\, $\d \to \infty$,
\end{center}
and prove the convergence of minimizers via $\Gamma$-convergence (cf.~\cite{DalMaso, Braides}) for the corresponding families of $\delta$-dependent functionals $\Fcal_\delta$. They consist of vectorial integrals of the form
\begin{equation}\label{eq:fcaldintro}
\Fcal_\d (u):= \int_{\O_\d} f(x,D_{\rho_\d} u)\,dx, 
\end{equation}
where the involved quantities are given as follows: The set $\Omega\subset \R^n$ is a bounded domain, $\O_\d = \O+B_\d(0)$ is its expansion by the horizon parameter, and the integrand $f:\R^n\times \R^{n\times m}\to \R$ is assumed to have $p$-growth for $p \in (1,\infty)$ and to be quasiconvex in its second argument. The space of admissible functions for~\eqref{eq:fcaldintro} is $H^{\rho,p,\d}_0(\O;\R^m)$, the natural nonlocal Sobolev space associated to the gradient $D_{\rho_\d}$ with a zero complementary-value conditions, meaning that the functions vanish in $\O^c$, see~Section~\ref{se: preliminaries} for more details. Let us now give a brief overview of our findings on the two limit passages $\d \to 0$ and $\d \to \infty$. \medskip

\textbf{(i) Localization via shrinking horizon $\delta\to 0$.} With the scaling factors $c_\d = \d^{-n}$ in~\eqref{eq:scalingintro}, which preserves the mass of the kernel $\rho_\d$, we confirm for each $u\in W^{1,p}(\R^n;\R^m)$ the convergence of the nonlocal gradients to the classical one, precisely,
\begin{align*}
D_{\rho_\d}u \to  \nabla u \quad \text{ in $L^p(\R^n;\R^{m\times n})$ as $\delta\to 0$,}
\end{align*} 
see Lemma~\ref{le:recovery}; for sufficiently smooth functions, the convergence is uniform and the optimal rate of convergence is explicitly determined by $\d^2$.

As our main result within (i), Theorem~\ref{th: Gamma convergence to 0} states the $\Gamma$-convergence of $(\Fcal_\d)_\d$ from~\eqref{eq:fcaldintro} with respect to the strong $L^p$-topology as $\delta\to 0$ to the limit functional $\Fcal_0$ given by
\[
\Fcal_0 (u) = \int_{\O} f(x, \nabla u)\,dx \quad  \text{ for $u \in W^{1,p}_0(\O;\R^m)$,}
\]
and provides along with this, the corresponding equi-coercivity of $(\Fcal_\d)_\d$; the latter constitutes the major novelty of Theorem~\ref{th: Gamma convergence to 0}, as explained below. Consequently, the minimizers of $\Fcal_\d$, which exist by Theorem~\ref{th:existence}, converge (up to subsequences) in $L^p$ to a minimizer of $\Fcal_0$.

The key technical ingredient for our proof of equi-coercivity for $(\Fcal_\d)_\d$ is the estimate~\eqref{eq:coercivityintro}, which can be seen as an enhanced Poincar\'e-type inequality. The proof of \eqref{eq:coercivityintro} uses the isotropic scaling from \eqref{eq:scalingintro} to identify the dependence of the Fourier symbol of $D_{\rho_\d}$ on $\d$. Together with the results on the Fourier symbol of $D_\rho$ for a fixed kernel \cite{BMS24}, we then deduce from the Mihlin-H\"{o}rmander theorem that the spaces $H^{\rho,p,\d}(\O;\R^m)$ do not change with $\d>0$ and that there is a $\delta$-independent constant $C>0$ such that
\begin{equation}\label{eq:coercivityintro}
\norm{u}_{H^{\sigma,p}(\Rn;\R^m)} \leq C \norm{D_{\rho_\d} u}_{L^p(\Rn;\Rmn)} \quad \text{for all $u \in H^{\rho,p,\d}(\O;\R^m)$ and $\d \in (0,1]$,}
\end{equation}
where $\sigma >0$ is related to the kernel $\rho$, see Theorem~\ref{th:comparison} and Corollary~\ref{cor:compactness}. 

To set our contribution into context with the existing literature, we mention that closely related localization results for nonlocal gradients in various relevant topologies can be found in \cite{MeS} or deduced from the results in the recent paper \cite{Arr23} by Arroyo-Rabasa, which addresses more general nonlocal first-order linear operators.
Mengesha \& Spector in~\cite[Theorem~1.7]{MeS} also present a first $\Gamma$-convergence result for scalar and convex variational problems in their setting. Beyond the fact that we consider more general quasiconvex integrands in the vectorial case, the only minor difference with our set-up (see~\eqref{Drho_intro} and~\eqref{eq:fcaldintro}) lies in the definition of the nonlocal gradient, where they consider interactions only between points within the domain $\Omega$. Accordingly, some of our arguments regarding the liminf-inequality and the construction of recovery sequences share similarities with~\cite{MeS}.
However, in contrast to our work,~\cite{MeS} does not contain any compactness results (uniformly in $\delta$) or equi-coercivity results, and, therefore, cannot guarantee the existence or the convergence of minimizers for the involved integral functionals. On the other hand, in a different setting of non-symmetric half-space nonlocal gradients for the case $p=2$, such compactness results have recently been obtained uniformly in $\d$ \cite{HMT24}. \medskip

\textbf{(ii) Fractional models via diverging horizon $\d\to \infty$.} For the regime of large horizons, we identify the scaling factors $c_\d = \overline{\rho}(1/\d)^{-1}$ in~\eqref{eq:scalingintro}, with $\overline{\rho}$ the radial representation of $\rho$, and assume additionally that the kernels $\rho_\d$ converge pointwise to a function $\rho_{\infty}$; the scaling factors ensure that $\rho_\infty$ is equal to 1 on the unit sphere. It turns out that the limit kernel $\rho_{\infty}$ will always be a fractional kernel, that is, 
\begin{align}\label{rhoinfty_intro}
\rho_\infty=|\cdot|^{-(n+s_\infty-1)}
\end{align} 
with some $s_\infty \in (0,1)$ characteristic for $\rho$, 
see~Lemma~\ref{le: fractional kernel as limit}. Even though this observation may be surprising, given that $\rho$ will generally not have a fractional singularity at the origin, it follows because of the fact that the radial representation of $\rho_\infty$ gains multiplicativity through the limit process, and is thus, a power function.

In view of~\eqref{rhoinfty_intro}, we then deduce in Proposition~\ref{prop:convergenceinfinite} the convergence of the nonlocal gradients to a Riesz fractional gradient. Precisely, it holds for $u \in W^{1,p}(\R^n;\R^n)$ that
\begin{align}\label{convergence_Ddelta_intro}
D_{\rho_\d} u\to  D^{s_\infty}u \quad \text{ in $L^p(\R^n;\R^{m\times n})$ as $\d \to \infty$.}
 \end{align}
Using a similar strategy as in (i), we can specify the dependence of the Fourier symbol of $D_{\rho_\d}$ on $\d$ and show the analogue of \eqref{eq:coercivityintro} for large $\d$, see~Proposition~\ref{prop:poincare2}. This facilitates along with~\eqref{convergence_Ddelta_intro} the proof of our main theorem on $\Gamma$-convergence and equi-coercivity of the family $(\Fcal_\d)_\d$ (cf.~\eqref{eq:fcaldintro}) as $\d \to \infty$; explicitly, Theorem~\ref{th:infty} yields the $\Gamma$-limit $\Gamma(L^p)\text{-}\lim_{\delta\to \infty} \Fcal_\d = \Fcal_\infty$
with
\begin{align}\label{Fcal_frac}
\Fcal_\infty(u) = \int_{\R^n} f(x,D^{s_\infty}u)\,dx  \qquad  \text{for $u \in H^{s_\infty,p}_0(\O;\R^m)$.}
\end{align}
 In particular, we have that the minimizers of $\Fcal_\d$ converge (up to subsequence) in $L^p$ to a minimizer of a variational integral depending on Riesz fractional gradients. 
 
Note that the resulting limit objects, that is, fractional functionals of the type~\eqref{Fcal_frac}, have been well-studied in the last years under different assumptions on the integrand, see~e.g.~\cite{ShS2015, BeCuMC, KreisbeckSchonberger, Sch23}. The aspects addressed include weak lower semicontinuity, relaxation, existence of minimizers, and Euler-Lagrange equations.\medskip

\color{black}
The prototypical example that illustrates our results is a truncated version of the Riesz fractional kernel, that is,
\[
\rho = \frac{w}{\abs{\,\cdot\,}^{n+s-1}}\ \quad  \text{for $s \in (0,1)$},
\]
where $w:\R^n\to [0, \infty)$ is a suitable smooth, radial cut-off function compactly supported in the closed unit ball of $\R^n$, 
cf.~Example~\ref{ex:kernels}\,a). Applying the two scaling choices of (i) and (ii) gives the scaled kernels 
\[
\rho_\d = \d^{s-1} \frac{w(\,\cdot\,/\d)}{\abs{\, \cdot\, }^{n+s-1}} \ 
 \quad \text{and} \quad \rho_\d = \overline{w}(1/\d)^{-1} \frac{w(\,\cdot\,/\d)}{\abs{\, \cdot\,}^{n+s-1}}, 
\]
respectively, where $\overline{w}$ denotes the radial representation of $w$. 
Both these kernels are supported in the ball $\overline{B_\delta(0)}$ of radius $\delta$ around the origin and give rise to the finite-horizon fractional gradients $D^s_\d$ studied in~\cite{BeCuMC23,BeCuMC23b,CuKrSc23,BeCuFoRa, KrS24};  we remark that in those references the dependence of the kernel on $\d$ is not made explicit, since the horizon was always considered fixed. As a consequence of the results in this paper, variational problems involving the gradients $D^s_\d$ are confirmed to approximate their local analogue with classical gradients in the limit of vanishing horizon. As $\delta \to \infty$, we justify the intuitive connection with the fractional case, which reflects an infinite range of interaction.\color{black}

Let us close the introduction by briefly pointing out some interesting further connections for broader context.
 We observe that for fractional gradients, 
  localization also occurs by letting the fractional index $s$ tend to $1$. 
 This was demonstrated for the Riesz fractional gradient $D^s$ in \cite{BeCuMC21} and for its finite-horizon version $D^s_\d$ in \cite{CuKrSc23, KrS24}. 
 Among the first works on nonlocal-to-local limit passages were those in the context of bond-based peridynamics~\cite{MeD, BeMCPe, AAB23, BBM01}, which successfully recover various classical models. However, as shown in \cite{BeCuMC20, MeS23}, there are classical energies 
 that cannot be obtained from double-integral functionals as $\delta$ goes to $0$, as opposed to the setting of this paper. 
Finally, the limit of diverging horizon is less common in the literature. Nevertheless, the convergence of finite-horizon versions of the fractional $p$-Laplacian has been established in \cite{BeO21a,BeO21b} for both $\delta\to 0$ and $\delta\to \infty$, recovering the classical and fractional $p$-Laplacian, respectively.\smallskip

\color{black}
The manuscript is organized as follows. After introducing in Section \ref{se: preliminaries} the nonlocal gradients and associated function spaces that we are going to work with, we specify the required conditions on the kernel $\rho$ and collect some technical tools and preliminary results. The core of the paper are Sections~\ref{se: localization} and \ref{se: convergence in infinity}, where we address the limit analysis of $\d \to 0$ and $\d \to \infty$ to recover the classical and fractional models, respectively. These two sections, which are each presented in a self-contained way, share a parallel structure:  First, showing the convergence of the varying horizon nonlocal gradients, then, compactness statements uniformly in $\d$, and, finally, the $\Gamma$-convergence of the energy functionals in \eqref{eq:fcaldintro}.

\color{black}
\section[Preliminaries]{Preliminaries} \label{se: preliminaries}
\subsection{Notation} 
 We write $\abs{x}=\left(\sum_{i=1}^n x_i^2\right)^{1/2}$ for the Euclidean norm of a vector $x=(x_1,\cdots,x_n)\in \R^n$ and $\abs{A}$ for the Frobenius norm of a matrix $A \in \Rmn$. The ball centered at $x \in \R^n$ and with radius $r>0$ is denoted by $B_r(x)=\{ y \in \R^n : \abs{x-y}<r\}$ and the distance between $x \in \R^n$ and a set $E\subset \R^n$ is written as $\dist(x,E)$. For an open set $\Omega \subset \R^n$ and $\d>0$, we define
\begin{align}\label{Omegadelta}
\Omega_\delta:= \Omega+B_\d(0)=\{x \in \R^n\,:\, \dist(x,\Omega) <\delta\}
\end{align}
and set $\O_{-\d}:=\{x \in \O \,:\, \dist(x,\O^c) > \d\}$. 
The complement of a set $E\subset \R^n$ is indicated by $E^c:=\R^n \setminus E$, its closure by $\overline{E}$, and its boundary by $\partial E$. We take
\[
\mathbbm{1}_E(x) = \begin{cases} 1 &\text{for} \ x \in E,\\
0 &\text{otherwise},
\end{cases}\qquad x\in \R^n,
\]
to be the indicator function of a set $E \subset \R^n$. The extension of a function $u:E \to \R$ to $\Rn$ as zero is sometimes explicitly denoted as $\mathbbm{1}_E\, u$. For a function $u:\Rn \to \R$, we denote its support by $\supp u$, and, if $u$ is Lipschitz continuous, its Lipschitz constant by $\Lip(u)$. 

For $U \subset \Rn$ open, we adopt the standard notation for the space of smooth functions with compact support $C_c^{\infty}(U)$, the Lebesgue space $L^p(U)$ and the Sobolev space $W^{1,p}(U)$ with $p \in [1,\infty]$. The spaces can be extended componentwise to vector-valued functions; the target space is explicitly mentioned in the notation, like, for example, $L^p(U;\R^m)$. We use the usual multi-index notation for partial derivatives $\partial^{\alpha}$ with $\alpha \in \N_0^n$. Our convention for the Fourier transform of $f \in L^1(\R^n)$ is
\begin{equation*}
	\widehat{f}(\xi)=\int_{\Rn} f(x) \, e^{-2\pi i x \cdot \xi} \, dx, \quad \xi \in \R^n,
\end{equation*}
see e.g.,~\cite{Gra14a} for more details. 

For real functions, we use the monotonicity properties of being increasing and decreasing in the non-strict sense.  
A function $f:\R\to \R$ is called almost decreasing if there is a $C>0$ such that $f(t) \geq Cf(s)$ for $t \leq s$, and an analogous definition holds for almost increasing. 
\color{black}
 For a radial function $p:\Rn \to \R$, we denote its radial representation by $\overline{p}:[0,\infty) \to \R$, i.e., $p(x) = \overline{p}(\abs{x})$ for $x \in \Rn$,   \color{black} and call $p$ radially decreasing or increasing, if its radial representation is decreasing or increasing, respectively.
 
  Finally, throughout the manuscript, we use generic constants, which may change from line to line without renaming.

\subsection[Nonlocal operators and function spaces]{Nonlocal gradients}\label{sec:nonlocalgradients}
We now introduce the key elements of our setting, that is, the nonlocal gradients for general kernels $\rho$ as recently studied in~\cite{BMS24}; for related work see~also~\cite{DGLZ, Delia, EGM22}, as well as~\cite{BeCuMC23,BeCuMC23b,CuKrSc23} on the special case of finite-horizon fractional gradients. \color{black}

Assume throughout that $\rho:\R^n \setminus \{0\} \to [0,\infty)$ is a radial kernel such that\smallskip
\begin{enumerate}[label = (H\arabic*)]
\setcounter{enumi}{-1}
\item \label{itm:h0} $\inf_{\overline{B_{\epsilon}(0)}} \rho >0 \ \text{for some $\epsilon>0$} \quad \text{and} \quad \rho\, \min\{1,|\cdot|^{-1}\} \in L^1(\R^n).$
\end{enumerate}
Under this condition, the kernel gives rise to an associated nonlocal gradient.
\color{black}
\begin{definition}[Nonlocal gradient]
The nonlocal gradient $D_{\rho} \phi:\R^n \to \R^n$ for $\phi\in C^{\infty}_c(\Rn)$ is given by
\begin{equation*}
D_\rho \phi(x) = \int_{\R^n} \frac{\phi(x)-\phi(y)}{\abs{x-y}}\frac{x-y}{\abs{x-y}}\rho(x-y)\,dy, \quad x \in \R^n.
\end{equation*}
\end{definition}

We collect here some key properties of $D_\rho$ that will be used later on. First of all, with $\overline{\rho}:(0,\infty) \to [0,\infty)$ the radial representation of $\rho$, i.e., $\rho=\overline{\rho}(\abs{\,\cdot\,})$, the nonlocal gradient can be written as the convolution of the classical gradient with the locally integrable function
\[
Q_\rho(x) := \int_{\abs{x}}^{\infty} \frac{\overline{\rho}(r)}{r}\,dr,  \quad x \in \R^n \setminus \{0\},
\]
that is,
\begin{align}\label{eq:translationsmooth}
D_\rho \phi &= Q_\rho * \nabla \phi = \nabla (Q_\rho * \phi) \quad \text{for all $\phi \in C_c^{\infty}(\R^n)$}, 
\end{align}
see~\cite[Propositions~2.6]{BMS24}.
When $\rho \in L^1(\R^n)$, then also $Q_\rho \in L^1(\R^n)$ and one obtains after taking the Fourier transform that
\begin{equation}\label{eq:drhofourier}
\widehat{D_\rho \phi}(\xi) = 2\pi i \xi \widehat{Q}_\rho(\xi) \widehat{\phi}(\xi) \quad \text{for $\xi \in \R^n$},
\end{equation}
see~\cite[Propositions~2.5\,$(iii)$ and 2.6]{BMS24}.  

An important property of the nonlocal gradient is the presence of a duality relation with the nonlocal divergence, as conveyed by the following integration by parts formula, cf.~\cite[Proposition~3.2]{BMS24}. 

\color{black}
\begin{lemma}[Integration by parts, \cite{BMS24}]
	Let $\f\in C^{\infty}_c(\Rn)$ and $\psi\in C^{\infty}_c(\Rn;\Rn)$. Then
	\begin{align}\label{eq:integration by parts}
		\int_{\Rn} D_{\rho}\f \cdot \psi \, dx =- \int_{\Rn} \f \diver_{\rho}\psi  \, dx,
	\end{align}
	where 
	\begin{equation*}
		\diver_{\rho} \psi(x):=\int_{\R^n} \frac{\psi(x)-\psi(y)}{\abs{x-y}}\cdot\frac{x-y}{\abs{x-y}}\rho(x-y)\,dy .
	\end{equation*}
\end{lemma}
Note that the integration over the whole space in~\eqref{eq:integration by parts} is sufficient for our purpose, since we will be working mainly with compactly supported functions; an alternative version of integration by parts over a bounded domain giving rise to boundary terms was recently proven in~\cite{BeCuFoRa}.

The previous lemma motivates a distributional definition of nonlocal gradients.
\begin{definition}[Weak nonlocal gradients] Let $u\in L^1(\Rn)+L^\infty(\Rn)$. We say that $V \in L^1_{\loc}(\Rn;\Rn)$ is the weak nonlocal gradient of $u$, and write $D_\rho u =V$, if
	\begin{equation*}
		\int_{\Rn} V \cdot\psi \, dx = - \int_{\Rn} u \diver_{\rho} \psi \, dx \quad \, \text{ for all } \, \psi \in C^{\infty}_c(\Rn; \Rn).
	\end{equation*}
\end{definition}

In analogy to classical Sobolev spaces, one introduces for $p \in (1,\infty)$ the $\rho$-nonlocal Sobolev spaces as 
\[
H^{\rho,p}(\R^n):=\{u \in L^p(\R^n)\,:\, D_\rho u \in L^p(\R^n;\R^n)\},
\]
endowed with the norm
\begin{equation} \label{eq: nonlocal norm}
	\norm{u}_{H^{\rho,p}(\R^n)} :=\norm{u}_{L^p(\R^n)} + \norm{D_\rho u}_{L^p(\R^n;\R^n)},
\end{equation}
see~\cite[Definition~3.4]{BMS24}. Note that these spaces can be equivalently characterized as the closure of $C_c^{\infty}(\R^n)$ under the norm in \eqref{eq: nonlocal norm} in light of \cite[Theorem~3.9\,(i)]{BMS24}. Additionally, we define for an open $\O \subset \R^n$ the subspaces
\[
H^{\rho,p}_0(\Omega):=\overline{C_c^{\infty}(\Omega)}^{\norm{\cdot}_{H^{\rho,p}(\R^n)}},
\]
where the elements of $C_c^{\infty}(\O)$ are interpreted as extended to $\Rn$ by zero.  If $\Omega$ is a bounded Lipschitz domain, $H_0^{\rho, p}(\Omega)$ agrees with the complementary-value space of the functions in $H^{\rho,p}(\R^n)$ that are zero in $\Omega^c$, see~\cite[Theorem~3.9\,(iii)]{BMS24}. Prescribed complementary values can be viewed as the nonlocal analogue of Dirichlet boundary conditions in the local setting.

\begin{example}[Riesz fractional gradient and Bessel potential spaces]\label{ex:Rieszfracgrad}
The special choice of kernel function
\begin{align}\label{eq:rhoskernel}
\rho^s:= \frac{1}{|\cdot|^{n+s-1}}\quad \text{ with $s \in (0,1)$}
\end{align}
gives rise to the Riesz $s$-fractional gradient $D^s:=D_{\rho^s}$,  given for $\varphi\in C_c^\infty(\R^n)$ by
\[
D^s \varphi(x)= \int_{\Rn} \frac{\varphi(x)-\varphi(y)}{|x-y|^{n+s}}\frac{x-y}{|x-y|}\, dy, \quad x\in \R^n,
\]
cf.~\cite{ShS2015,ShS2018}. Commonly, $D^s$ features a normalization constant $c_{n,s}$, which we omit here for the sake of a cleaner presentation in Section~\ref{se: convergence in infinity}. 
The associated nonlocal Sobolev space $H^{\rho^s\hspace{-0.07cm},p}(\Rn)$ coincides with the Bessel potential space $H^{s,p}(\Rn)$ as shown in~\cite[Theorem~1.7]{ShS2015}. 
A property we will often exploit is that
\begin{align}\label{embeddingBessel}
\text{$H^{s,p}(\Rn)$ 
is compactly embedded into $L^p_{\loc}(\Rn)$, }
\end{align}
see~e.g.,~\cite[Theorem 2.2]{ShS2018} or \cite[Theorem 2.3]{BeCuMC}. Moreover, we set $H^{s,p}_0(\Omega)=\{u\in H^{s,p}(\R^n): u=0 \ \text{a.e.~in $\Omega^c$}\}$.
\end{example}

Besides the previously introduced hypothesis \ref{itm:h0} on the kernel $\rho$, used for the definition of the nonlocal gradient, we require a few more properties in order to have a wider variety of technical tools, such as compact embeddings and Poincar\'{e} inequalities, at our disposal. In accordance with~\cite{BMS24} (see also \cite[Remark~4.1]{BMS24}), we make the following assumptions: 

\color{black}
Let $\epsilon$ be as in (H0), $\nu >0$ and $0<\sigma \leq \gamma <1$. \smallskip

\begin{enumerate}[label = (H\arabic*)]

\item \label{itm:h1} The function $f_\rho:(0,\infty)\to \R, \ r \mapsto r^{n-2}\overline{\rho}(r)$ is decreasing on $(0,\infty)$ and $r \mapsto r^\nu f_\rho(r)$ is decreasing on $(0,\epsilon)$;

\item \label{itm:derivatives} $f_\rho$ is smooth outside the origin and for every $k\in \N$ there exists a $C_k>0$ with 
\[
\abslr{\frac{d^k}{dr^k}f_\rho(r)} \leq C_k \frac{f_\rho(r)}{r^k} \quad \text{for $r\in (0, \epsilon)$};
\]

\item \label{itm:lower} the function $r \mapsto r^{n+\sigma-1}\overline{\rho}(r)$ is almost decreasing on $(0,\epsilon)$; 

\item \label{itm:upper} the function $r \mapsto r^{n+\gamma-1}\overline{\rho}(r)$ is almost increasing on $(0,\epsilon)$.
\end{enumerate}\smallskip

Most of the time, we will not work directly with these hypotheses, but instead make use of the results and tools proven in \cite{BMS24}; we refer to that paper for a more detailed discussion of the assumptions \ref{itm:h0}-\ref{itm:upper}.  The Riesz potential kernel from \eqref{eq:rhoskernel} satisfies all these properties, as one can easily check. Beyond that, we list here a few examples with compactly supported kernels from \cite[Example~5.1]{BMS24} that fit into the setting. These will be revisited also in the later sections to illustrate our findings. 

\color{black}
\begin{example}[Selected kernel functions $\rho$]\label{ex:kernels}
Let $w \in C_c^{\infty}(\Rn)$ be a non-negative radial function with $w(0)>0$. \smallskip

a) Let $s \in (0,1)$ and suppose that $w/\abs{\,\cdot\,}^{1+s}$ is radially decreasing. Then,
\begin{align*}\label{ex:finitehorizonkernel}
\rho(x) = \frac{w(x)}{\abs{x}^{n+s-1}}, \quad x\in \R^n\setminus \{0\}, 
\end{align*}
satisfies \ref{itm:h0}-\ref{itm:upper} with $\sigma=\gamma=s$. The associated nonlocal gradient $D_\rho$ is referred to as a finite-horizon fractional gradient. In fact, it holds that $H^{\rho,p}(\Rn)=H^{s,p}(\Rn)$ with equivalent norms by \cite[Proposition~3.10]{BMS24}. 

\smallskip

b) Let $s \in (0,1)$ and $\kappa\in \{-1,1\}$. If $\supp(w) \subset \overline{B_1(0)}$ and $w\log^\kappa (1/\abs{\,\cdot\,})/\abs{\, \cdot\,}^{1+s}$ is radially decreasing, then the kernel function given by
\[
\rho(x) = \frac{w(x)\log^\kappa(1/\abs{x})}{\abs{x}^{n+s-1}}, \quad\text{ $x\in \R^n\setminus \{0\}$,} 
\]
satisfies \ref{itm:h0}-\ref{itm:upper} with $\sigma=s$ and any $\gamma \in (s,1)$  if $\kappa=1$ and with any $\sigma \in (0,s)$ and $\gamma = s$ if $\kappa=-1$.  \smallskip

c) Consider a smooth function $s:[0,\infty) \to (0,1)$ and let $w/\abs{\,\cdot\,}^{1+s(\abs{\, \cdot\,})}$ be radially decreasing. Then, 
\[
\rho(x)=\frac{w(x)}{\abs{x}^{n+s(\abs{x})-1}}, \quad x\in \R^n\setminus\{0\}, 
\]
is a kernel with spatially varying fractional parameter satisfying~\ref{itm:h0}-\ref{itm:upper} with $\sigma = \min_{[0,\epsilon]} s$ and $\gamma = \max_{[0,\epsilon]} s$ for any $\epsilon >0$.
\end{example}

\color{black}
The following auxiliary result from \cite[Lemma~4.3, 4.10 and 7.1]{BMS24} will be exploited in Sections~\ref{se: localization} and \ref{se: convergence in infinity} to prove compactness results uniformly in the horizon parameter. It provides bounds on the Fourier transform of $Q_\rho$ and its derivatives in terms of the radial representation of $\rho$.
\color{black}
\begin{lemma}[Estimates on $\widehat{Q}_\rho$ and its derivatives, {\cite{BMS24}}]\label{lem:bms}
Let $\rho:\R^n\setminus\{0\}\to [0, \infty)$ be a radial kernel with compact support satisfying \ref{itm:h0}-\ref{itm:upper}. Then $\widehat{Q}_\rho$ is smooth, positive, and there exists a constant $C>0$ such that
\begin{align}\label{eq:bms1}
\frac{1}{C} \frac{\bar{\rho}(1/\abs{\xi})}{\abs{\xi}^n} \leq \widehat{Q}_\rho(\xi) \leq C \frac{\bar{\rho}(1/\abs{\xi})}{\abs{\xi}^n} \quad \text{for all $\abs{\xi} \geq 1/\epsilon$}.
\end{align}
Moreover, for every $\alpha \in \N_0^n$, one has
\begin{align}\label{eq:bms2}
\absB{\partial^{\alpha} \widehat{Q}_\rho(\xi)} \leq C_{\alpha} \abs{\xi}^{-\abs{\alpha}} \abslr{\widehat{Q}_\rho(\xi)} \quad \text{for all $\xi \not = 0$}
\end{align}
with constants $C_\alpha>0$.
\end{lemma}

Finally, Poincar\'e-type inequalities will be indispensible tools for our analysis. We present here a particular consequence of~\cite[Theorem~4.11]{BMS24}, that suffices for our setting.

\color{black}
\begin{lemma}[Poincar\'{e} inequalities and compact embedding, {\cite{BMS24}}]\label{lem:compactness}
Let $\O \subset \R^n$ be open and bounded and suppose that the radial kernel function $\rho$ satisfies \ref{itm:h0}-\ref{itm:upper} and has compact support. Then there is a $C>0$ such that
\begin{align}\label{Poincare}
\norm{u}_{L^p(\Rn)} \leq C \norm{D_\rho u}_{L^p(\Rn;\Rn)} \quad \text{for all $u \in H^{\rho,p}_0(\O)$},
\end{align}
and $H^{\rho,p}_0(\O)$ is compactly embedded into $L^p(\Rn)$. 
\end{lemma}

In fact, comparing $\rho$ as in the previous lemma with the kernel from Example~\ref{ex:kernels}\,a) with $s=\sigma$ leads to a stronger estimate that will be utilized several times in what follows. Precisely, by using \ref{itm:lower} and  \cite[Theorem~7.2]{BMS24}, we find that there is a $C>0$ such that 
\begin{align}\label{estimateHsigmap}
\norm{u}_{H^{\sigma,p}(\Rn)} \leq C \norm{D_\rho u}_{L^p(\Rn;\Rn)} \quad \text{for all $u \in H^{\rho,p}_0(\O)$}.
\end{align}

\begin{remark}[Relaxed assumptions on $\rho$]\label{rem:relaxed}
Note that according to~\cite[Proposition~3.10]{BMS24}, there is an equivalence between the function spaces and Poincar\'{e} inequalities associated to kernels that agree around the origin. Hence,~\eqref{Poincare} and~\eqref{estimateHsigmap} hold even when the smoothness in \ref{itm:h1} only holds locally, or when the assumption of $\rho$ having compact support is dropped. For example, one could replace the function $w$ in Example~\ref{ex:kernels}\,a) by an indicator function $\mathbbm{1}_{B_\delta(0)}$ with $\delta>0$ or by an exponentially decaying function $e^{-\alpha|\, \cdot\,|}$ with $\alpha>0$, which leads to the truncated and tempered fractional kernel of~\cite[Examples~2 and~3]{EGM22}, respectively.
\end{remark}

\subsection[Existence theory for nonlocal variational problems]{Existence theory for nonlocal variational problems}\label{sec:existence}

Let us address next the solvability of vectorial variational problems involving the nonlocal gradients as introduced in the previous section. Besides being of general interest, the existence statement of Theorem~\ref{th:existence} is needed below to conclude the convergence of minimizers for the variational problems with varying horizon in Sections~\ref{se: localization} and \ref{se: convergence in infinity}.
We remark that the results presented here are new in this generality, but can be derived by following closely the techniques of \cite{CuKrSc23}, where the direct method in the calculus of variations is applied to the special case of functionals depending on finite-horizon fractional gradients. The adaptation of the proofs is straightforward and left to the reader.

Throughout this section, we assume that $p \in (1,\infty)$, $\O \subset \Rn$ is a bounded Lipschitz domain, and the kernel $\rho$ satisfies \ref{itm:h0}-\ref{itm:upper} and has compact support. The following result, which allows us to translate the nonlocal gradients into classical gradients, can be proven by extending \eqref{eq:translationsmooth} via density as in \cite[Theorem~2\,$(i)$]{CuKrSc23}.
\color{black}
\begin{lemma}[From nonlocal to local gradients]\label{le:translation}
The linear map $\Qcal_\rho:H^{\rho,p}(\R^n) \to W^{1,p}(\R^n),$  $u \mapsto Q_\rho * u$ is bounded and it holds for all $u \in H^{\rho,p}(\R^n)$ that
\[
D_\rho u= \nabla (\Qcal_\rho u).
\]
\end{lemma}

Another ingredient is the strong convergence of nonlocal gradients in the complement of $\O$, which follows as in \cite[Lemma~3]{CuKrSc23} by utilizing the compact embedding of $H^{\rho,p}_0(\O)$ into $L^p(\Rn)$ (see Lemma~\ref{lem:compactness}) and the Leibniz rule in \cite[Lemma~3.8]{BMS24}. 
\color{black}
\begin{lemma}[Strong convergence in the complement]
Let $(u_j)_j \subset H^{\rho,p}_0(\O)$ be a sequence that converges weakly to $u$ in $H^{\rho,p}(\Rn)$. Then, for any $\eta>0$ it holds that
\[
D_\rho u_j \to D_\rho u \quad \text{in $L^p((\O_\eta)^c;\Rn)$ as $j\to \infty$},
\]
recalling the definition $\O_\eta=\O+B_\eta(0)$, see~\eqref{Omegadelta}.
\end{lemma}

With these two technical tools and the Poincar\'{e} inequality from Lemma~\ref{lem:compactness} at hand, one can argue as in the sufficiency part of \cite[Theorem~5]{CuKrSc23} and \cite[Corollary~2]{CuKrSc23} to obtain the existence of minimizers for vectorial variational problems with quasiconvex integrands. \color{red} 
\color{black}
\begin{theorem}[Existence of minimizers]\label{th:existence}
Let $\d>0$ be such that $\supp \rho = \overline{B_\d(0)}$ and let $f:\O_\d \times \Rmn \to \R$ be a Carath\'{e}odory integrand such that 
\[
c\abs{A}^p-C\leq f(x,A) \leq C(1+\abs{A}^p) \quad \text{for a.e.~$x \in \Omega_\d$ and all $A \in \Rmn$}.
\]
If $A \mapsto f(x,A)$ is quasiconvex for a.e.~$x \in \O$, then the functional 
\begin{align}\label{Fcal}
\Fcal:H^{\rho,p}_0(\O;\R^m) \to \R, \quad \Fcal(u):= \int_{\O_\d} f(x,D_\rho u)\,dx
\end{align}
admits a minimizer.
\end{theorem}
Note that taking $\d>0$ such that $\supp \rho = \overline{B_\d(0)}$ ensures that $D_\rho u$ is zero in $\O_\d^c$ for all $u \in H^{\rho,p}_0(\O)$. Hence, the functional $\Fcal$ in~\eqref{Fcal} defined as an integral over the bounded set $\Omega_\d$ captures all the non-trivial parts of $D_\rho$. 
 
Quasiconvexity, which is well-known to characterize the weak lower semicontinuity of integral functionals in the classical case \cite{Morrey, Dac08}, is indeed the natural convexity notion also in the context of variational integrals depending on nonlocal gradients. This observation can be seen as a generalization of~\cite[Theorem~1.1]{KreisbeckSchonberger} and~\cite[Theorem~5]{CuKrSc23} and relies on Lemma~\ref{le:translation} and the following inverse translation operator.
\begin{lemma}[From local to nonlocal gradients]
There is a bounded linear operator $\Pcal_\rho :W^{1,p}(\R^n) \to H^{\rho,p}(\R^n)$ such that $\Pcal_\rho=(\Qcal_\rho)^{-1}$. In particular, for all $v \in W^{1,p}(\R^n)$ we have
\[
 \nabla v= D_\rho (\Pcal_\rho v).
\]
\end{lemma}
\begin{proof}
We define the operator
\[
\Pcal_\rho:\Scal(\R^n) \to \Scal(\R^n), \quad \Pcal_\rho v :=\left(v/\widehat{Q}_\rho\right)^{\vee},
\]
which is well-defined given that $1/\widehat{Q}_\rho$ is smooth and polynomially bounded by Lemma~\ref{lem:bms} and \ref{itm:upper}. It is also clear that $\Pcal_\rho=(\Qcal_\rho)^{-1}$ and $D_\rho \circ\Pcal_\rho = \nabla$ on the space $\Scal(\R^n)$, so it is sufficient to prove that $\Pcal_\rho$ extends to a bounded operator from $W^{1,p}(\R^n)$ to $L^p(\R^n)$. By Lemma~\ref{lem:bms}, \ref{itm:lower} and the Mihlin-H\"ormander theorem (cf.~e.g.~\cite[Theorem~6.2.7]{Gra14a}), it can be verified that $\abss{\cdot}^{\sigma-1}/\widehat{Q}_\rho$ is an $L^p$-multiplier, where $\langle \xi \rangle:=\sqrt{1+|\xi|^2}$ for $\xi\in \R^n$. Hence, arguing as in \cite[Section~2.3]{KrS24}, we find that $\Pcal_\rho$ extends to a bounded operator from $H^{1-\sigma,p}(\R^n)$ to $L^p(\R^n)$ and, in particular, it is also bounded from $W^{1,p}(\R^n)$ to $L^p(\R^n)$.
\end{proof}
The following can now be proven as in \cite[Remark~8]{CuKrSc23}, given that we have both translation operators $\Qcal_\rho,\Pcal_\rho$, the Leibniz rule from \cite[Lemma~3.8]{BMS24}, and the compact embedding of $H^{\rho,p}_0(\O)$ into $L^{\infty}(\R^n)$ for $p>n/\sigma$ (cf.~\cite[Theorem~6.5]{BMS24}).
\begin{corollary}[Nonlocal representation of quasiconvexity] 
Let $\d>0$ be such that $\supp \rho = \overline{B_\d(0)}$. A continuous function $h:\Rmn \to \R$ is quasiconvex if and only if
\[
f(A) \leq \frac{1}{\abs{\O_\d}}\int_{\O_\d} f(A+D_\rho u)\,dx \quad \text{for all $A \in \Rmn$ and  $u \in H^{\rho,\infty}_0(\O;\R^m)$},
\]
with $H^{\rho,\infty}_0(\O;\R^m):=\{u \in L^{\infty}(\R^n;\R^m)\,:\, D_\rho u \in L^{\infty}(\R^n;\Rmn), \ u=0 \ \text{a.e.~in $\O^c$}\}$.
\end{corollary}

\color{black}

\subsection{Scaled kernels} We introduce here the setting and notations for varying horizon nonlocal gradients obtained via scaling of a fixed nonlocal gradient, as they will be used in the limits of vanishing and diverging horizon in Sections~\ref{se: localization} and \ref{se: convergence in infinity}. 

Our starting point is a radial kernel $\rho$ that satisfies \ref{itm:h0}-\ref{itm:upper} and is normalized in the sense that
\begin{equation}\label{eq:normalized}
\supp \rho = \overline{B_1(0)} \quad \text{and} \quad \int_{\Rn} \rho \,dx =n.
\end{equation}
One can then compute that
\begin{align}\label{normalization_Qrho}
\int_{\R^n} Q_\rho\,dx=1, \quad \text{ or equivalently, } \quad \widehat{Q}_\rho(0)=1.
\end{align}
Notice that the kernels from Example~\ref{ex:kernels} can all be rescaled and normalized to satisfy \eqref{eq:normalized}.

The rescaled family of kernels $(\rho_\d)_\d$ for horizons $\d>0$ is then defined by
\begin{align}\label{eq:scaledprelim}
\rho_\d (x) = c_\d \rho\left(\frac{x}{\d}\right), \quad x \in \Rn,
\end{align}
with $(c_\d)_\d \subset (0,\infty)$ a suitable sequence of scaling factors. Precisely, they are chosen as $c_\d = \d^{-n}$ for the limit $\d \to 0$ and as $c_\d = \overline{\rho}(1/\d)^{-1}$ for the limit $\d \to \infty$.

We collect here a few general observations about the rescaled kernels and associated gradients. First,  it follows that $\supp \rho_\d = \overline{B_\d(0)}$, which makes $D_{\rho_\d}$ a nonlocal gradient with horizon $\d$; in particular, the initial gradient $D_\rho$ corresponds with the gradient $D_{\rho_1}$ with horizon distance equal to $1$.  Moreover, the rescaling preserves the key properties of the kernel function, that is, for any $\d >0$, the kernel $\rho_\d$ also satisfies \ref{itm:h0}-\ref{itm:upper}. This makes all the results in the previous sections applicable to these kernels as well, in particular, the existence result of Theorem~\ref{th:existence}. 
Finally, we observe that the kernel associated to $\rho_\d$ satisfies
\begin{equation}\label{eq:qrhoscaling}
Q_{\rho_\d} = c_\d Q_\rho \left( \frac{\cdot}{\d}\right) \quad \text{and} \quad \widehat{Q}_{\rho_\d} = c_\d \d^{n}\widehat{Q}_\rho (\d \cdot),
\end{equation}
where we have used \cite[Proposition~2.3.22\,(7)]{Gra14a} for the interaction between scaling and Fourier transforms.

To highlight the dependence on the horizon parameter, we denote the spaces associated to $D_{\rho_\d}$ by 
\[
H^{\rho,p,\d}(\Rn) = \{ u \in L^p(\Rn)\,:\, D_{\rho_\d} u \in L^p(\Rn;\Rn)\},
\]
and similarly for $H^{\rho,p,\d}_0(\O)$; it holds specifically that $H^{\rho, p}(\R^n) = H^{\rho, p, 1}(\R^n)$.


\color{black}
\section[Localization when $\d\to 0$]{Localization when $\d\to 0$} \label{se: localization}

This section is devoted the localization process, that is, to the asymptotic analysis in the limit of vanishing horizon. 
We start by showing that the suitably scaled nonlocal gradients converge to the classical one as $\d\to 0$, and subsequently prove compactness results uniformly in the horizon parameter $\d$. Finally, we utilize these tools to establish the $\Gamma$-convergence of integral functionals depending on scaled nonlocal gradients to their local counterparts as the horizon tends to zero.

For this analysis, we fix a radial kernel $\rho$ that satisfies \ref{itm:h0}-\ref{itm:upper} and \eqref{eq:normalized}, and consider for $\delta \in (0,1]$, the scaled kernels
\[
\rho_\delta = \frac{1}{\delta^n}\rho\left(\frac{\,\cdot\,}{\delta}\right) \quad \text{and} \quad  Q_{\rho_\delta}= \frac{1}{\delta^n}Q_{\rho}\left(\frac{\,\cdot\,}{\delta}\right),
\]
which corresponds to \eqref{eq:scaledprelim} with the scaling factors $c_\d = \d^{-n}$. Observe that this choice of scaling preserves the normalizations $\int_{\R^n} \rho_\delta\,dx =n$ and $\int_{\R^n} Q_{\rho_\d}\,dx=1$ for each $\delta \in (0,1]$, and \eqref{eq:qrhoscaling} specifies to
\begin{equation}\label{eq:qrhofourierloc}
\widehat{Q}_{\rho_\d} = \widehat{Q}_\rho(\d \,\cdot\,).
\end{equation}
Throughout this section, we take $p \in (1,\infty)$ and assume $\O$ to be a bounded Lipschitz domain.

\subsection[Localization of the nonlocal gradient]{Localization of the nonlocal gradient}
Here, we present the convergence of the scaled nonlocal gradients to the classical gradient. Starting with the case of smooth functions, which features an explicit convergence rate, we subsequently extend the analysis to Sobolev functions on bounded domains and the whole space $\Rn$.  In the case of bounded domains, the nonlocal gradient is defined on a smaller set than the classical one, but this difference vanishes as $\d \to 0$. Closely related localization results can be found in \cite{MeS, Arr23}.
\color{black}
\begin{lemma}[Localization of nonlocal gradients]\label{le:recovery}
The following statements hold:
\begin{itemize}
	\item[$(i)$] For each $\phi \in C_c^{\infty}(\R^n)$ and for all $\d \in (0,1]$, one has that
	\begin{align}\label{eq:quadratic convergence}
	\norm{D_{\rho_\d} \phi - \nabla \phi}_{L^{\infty}(\R^n;\R^n)} \leq \d^2 \Lip(\nabla^2 \phi).
	\end{align}
 In particular, $D_{\rho_\d} \phi \to \nabla \phi$ uniformly on $\R^n$ as $\d \to 0$. \medskip
	
	\item[$(ii)$] For each $u \in W^{1,p}(\O)$, it holds that
	\[
\mathbbm{1}_{\O_{-\d}} D_{\rho_\d}u   \to \nabla u \quad \text{in } L^p(\O;\R^n) \text{ as } \d \to 0;
\]
recall that $\O_{-\d}:=\{x \in \O \,:\, \dist(x,\O^c) > \d\}$.\medskip

	\item[$(iii)$] For each $u \in W^{1,p}(\Rn)$, one has that $u\in H^{\rho_\d, p}(\R^n)$ for all $\delta\in (0,1]$, and
	\[
	D_{\rho_\d}u  \to \nabla u \quad \text{in } L^p(\Rn;\R^n) \text{ as }\d \to 0.
	\]
\end{itemize}
\end{lemma}

\begin{proof}
\textit{Part $(i)$:} Let $\psi \in C_c^{\infty}(\Rn;\Rn)$. Then the multivariate version of Taylor's theorem with integral remainder shows for $x \in \Rn$ that
\begin{align*}
	\abs{Q_{\rho_\d}*\psi(x)- \psi(x)} &= \abslr{\int_{B_\d(x)} Q_{\rho_\d}(x-y)\bigl(\psi(y) - \psi(x)\bigr)\,dy}\\
	&\leq \abslr{\int_{B_\d(x)} Q_{\rho_\d}(x-y)\nabla \psi(x)  (y-x)\,dy}\\
	&\qquad +\int_{B_\d(x)} Q_{\rho_\d}(x-y)\abslr{\int_0^1\bigl(\nabla \psi(x+t(y-x))-\nabla\psi(x)\bigr) (y-x)\,dt}\,dy\\
	&\leq \d^2 \Lip(\nabla \psi)\int_{B_\d(x)} Q_{\rho_\d}(x-y)\,dy= \d^2 \Lip(\nabla \psi),
\end{align*}
where we have used that $\norm{Q_{\rho_\d}}_{L^1(\R^n)}=1$, and also the radiality of $Q_{\rho_\d}$ to cancel the term in the second line. Applying this estimate with $\psi = \nabla \phi$ for any $\varphi\in C_c^\infty(\R^n)$ proves the claim in light of \eqref{eq:translationsmooth}. \smallskip

\textit{Part $(ii)$}: Since $u \in W^{1,p}(\O)$, the nonlocal gradient $D_{\rho_\d}u$ is well-defined in $\O_{-\d}$ and coincides with $Q_{\rho_\d}*\nabla u$ on this set (cf.~\cite[Proposition~3.5]{BMS24}). \color{black} Let $j\in \N$ and choose $\phi_j \in C_c^{\infty}(\R^n)$ such that 
\begin{align*}
\norm{\phi_j - u}_{W^{1,p}(\O)} \leq \frac{1}{j}, 
\end{align*}
which is possible as $\O$ is a bounded Lipschitz domain by assumption. Additionally, we can choose $\d=\d(j)$ small enough in light of Part $(i)$ such that 
\[
\norm{D_{\rho_\d}\phi_j - \nabla \phi_j}_{L^p(\O;\R^n)} \leq \frac{1}{j} \quad \text{and} \quad \norm{\nabla u}_{L^p(\O\setminus \O_{-\d};\Rn)} \leq \frac{1}{j}.
\]
The previous estimates along with~\eqref{eq:translationsmooth} then imply
\begin{align*}
\norm{\mathbbm{1}_{\O_{-\d}} D_{\rho_\d}u-\nabla u}_{L^p(\O;\Rn)} &\leq \norm{D_{\rho_\d}u-D_{\rho_\d}\phi_j}_{L^p(\O_{-\d};\Rn)}+\norm{D_{\rho_\d}\phi_j-\nabla\phi_j}_{L^p(\O_{-\d};\Rn)} \\
& \qquad +\norm{\nabla \phi_j - \nabla u}_{L^p(\O_{-\d};\Rn)}+\norm{\nabla u}_{L^p(\O\setminus \O_{-\d};\Rn)}\\\
& \leq \norm{Q_{\rho_\delta} \ast \nabla u - Q_{\rho_\delta}\ast \nabla \phi_j}_{L^p(\O_{-\d};\Rn)} + \frac{3}{j} \\ & \leq \norm{Q_{\rho_\d}}_{L^1(\Rn)} \norm{\nabla u - \nabla \phi_j}_{L^p(\O;\Rn)} + \frac{3}{j} \leq \frac{4}{j},
\end{align*}
where the last line is due to Young's convolution inequality.
 \smallskip

\color{black}
\textit{Part $(iii)$:} This follows with similar arguments as in Part\,$(ii)$ or, alternatively, from \cite[Theorem~C]{Arr23} with $\Acal=\nabla$.
\end{proof}

\begin{remark} \label{rem: localization of NL gradients}
a) In view of estimate~\eqref{eq:quadratic convergence}, $D_{\rho_\d}$ converges to $\nabla$ quadratically in $\d$, given that $\nabla^2 \phi$ is Lipschitz continuous. More generally, if $\varphi$ is twice differentiable such that $\nabla^2 \phi$ is $\alpha$-H\"{o}lder continuous with $\alpha\in (0,1]$, a similar argument induces the convergence rate $\delta^{1+\alpha}$, while for a differentiable $\varphi$ with $\alpha$-H\"{o}lder continuous gradient, convergence takes place at a rate of $\d^{\alpha}$. \smallskip

b) Our $\Gamma$-convergence result in Section~\ref{subsec:Gamma delta0} is formulated for admissible functions with prescribed  Dirichlet conditions in the complement of $\O$. 
Therefore, Lemma~\ref{le:recovery}\,$(iii)$ is sufficient for these purposes. However, the sharper result in Part\,$(ii)$ for bounded domains can be useful in the future for studying vanishing-horizon limits in more general settings, such as the Neumann-type problems considered in \cite{KrS24}. 
\end{remark}

\color{black}

\subsection{Compactness uniformly in $\delta \in (0,1]$}
In this section, we establish a compactness result for the nonlocal gradients that hold uniformly in the horizon parameter. 
The following theorem, which is also interesting in its own right (cf.~\eqref{eq:spaces delta independent} below), serves as a technical basis by providing a comparison between the norms of the nonlocal gradients with different horizons; this includes also the classical gradient, denoted for consistency by $D_{\rho_0}:=\nabla$. Our proof relies  on Fourier multiplier theory and takes inspiration from the one of \cite[Theorem~7.2]{BMS24} for comparing Sobolev spaces associated to different nonlocal gradients. 
\color{black}

\begin{theorem}[Comparison between scaled nonlocal gradients]\label{th:comparison}
Let $\bar{\delta}>0$ and $(\delta_1,\delta_2) \in [\bar{\delta},1]\times[0,1]$. Then, there exists a constant $C=C(\rho,n,p,\bar{\delta})>0$ such that
\[
\norm{D_{\rho_{\delta_1}}\varphi}_{L^p(\R^n;\R^n)} \leq C \norm{D_{\rho_{\delta_2}}\varphi}_{L^p(\R^n;\R^n)} \quad \text{for all $\varphi \in C_c^{\infty}(\R^n)$}.
\]
\end{theorem}

\begin{proof} Define the function $m_{\delta_1, \delta_2}:\R^{n}\setminus \{0\}\to (0, \infty)$ by
	\[
	m_{\delta_1, \delta_2}(\xi) := \frac{\widehat{Q}_\rho(\delta_1 \xi)}{\widehat{Q}_\rho(\delta_2 \xi)},
	\]
recalling that $\widehat{Q}_\rho$ is non-negative. 
Then, we find in view of \eqref{eq:drhofourier} and \eqref{eq:qrhofourierloc} that 
	\[
	\widehat{D_{\rho_{\delta_1}}\varphi} = m_{\delta_1, \delta_2}\widehat{D_{\rho_{\delta_2}}\varphi}. 
	\]for every $\varphi\in C_c^\infty(\R^n)$; in particular, the case $\d_2=0$ is valid given that $\widehat{Q}_{\rho}(0) =1$.
It now suffices to show with the help the Mihlin-H\"{o}rmander theorem (cf.~e.g.~\cite[Theorem~6.2.7]{Gra14a}) that $m_{\delta_1, \delta_2}$ are $L^p$-multipliers with estimates independent of $\delta_1$ and $\delta_2$. To this aim, we need to prove that for every $\alpha \in \N_0^n$ with $\abs{\alpha} \leq \tfrac{n}{2}+1$, 
	\begin{align}\label{eq:mihlin}
		\absb{\partial^{\alpha} m_{\delta_1, \delta_2}(\xi)} \leq C \abs{\xi}^{-\abs{\alpha}} \quad \text{for all $\xi \not =0$}
	\end{align}
with a constant $C>0$ depending only on $n, \rho$ and $\bar{\delta}$. 

 To this aim, note that the second part of Lemma~\ref{lem:bms} together with the Leibniz and quotient rules for differentiation 
 imply
	\[
	\abs{\partial^{\alpha} m_{\delta_1, \delta_2}(\xi)} \leq C \abs{\xi}^{-\abs{\alpha}} \abs{m_{\delta_1, \delta_2}(\xi)} \quad \text{for all $\xi \not =0$}
	\]
	with $C=C(n)>0$.
	\color{black}
	Therefore, it only remains to verify \eqref{eq:mihlin} for $\alpha =0$, that is, we need to show $m_{\delta_1, \delta_2}$ is uniformly bounded independent of $\delta_1$ and $\delta_2$. We prove this by distinguishing two cases.  \smallskip
	
	\textit{Case 1: $\delta_1 \geq \delta_2$.} For $0<\abs{\xi} \leq \tfrac{1}{\delta_2\epsilon}$ with $\eps>0$ the parameter in the hypotheses (H0)-(H4), one can estimate
	\[
	\absb{m_{\delta_1, \delta_2}(\xi)}=\abslr{\frac{\widehat{Q}_\rho(\delta_1 \xi)}{\widehat{Q}_\rho(\delta_2 \xi)}} \leq\textstyle \bigl(\min_{\overline{B_{1/\epsilon}(0)}}\widehat{Q}_\rho\bigr)^{-1},
	\]
considering that $\norm{\widehat{Q}_\rho}_{L^\infty(\R^n)}\leq \norm{Q_\rho}_{L^1(\R^n)}=1$ by~\eqref{normalization_Qrho}.
On the other hand, we infer for $\abs{\xi} \geq \tfrac{1}{\delta_2\epsilon}$ from Lemma~\ref{lem:bms} that
	\begin{align*}
|m_{\delta_1, \delta_2}(\xi)| 
&=\abslr{\frac{\widehat{Q}_\rho(\delta_1 \xi)}{\widehat{Q}_\rho(\delta_2 \xi)}} \leq C\left(\frac{\delta_2}{\delta_1}\right)^n \frac{\overline{\rho}\Bigl(\frac{1}{\delta_1\abs{\xi}}\Bigr)}{\overline{\rho}\Bigl(\frac{1}{\delta_2\abs{\xi}}\Bigr)} = C\left(\frac{\delta_2}{\delta_1}\right)^{1-\gamma} \frac{\overline{\rho}\Bigl(\frac{1}{\delta_1\abs{\xi}}\Bigr) \Bigl(\frac{1}{\delta_1\abs{\xi}}\Bigr)^{n+\gamma-1}}{\overline{\rho}\Bigl(\frac{1}{\delta_2\abs{\xi}}\Bigr)\Bigl(\frac{1}{\delta_2\abs{\xi}}\Bigr)^{n+\gamma-1}}  \\ &\leq  C \left(\frac{\delta_2}{\delta_1}\right)^{1-\gamma} \leq C\bar{\delta}^{\gamma-1},
	\end{align*}
	where the second inequality uses the almost monotonicity in~\ref{itm:upper}. \smallskip
	
	\textit{Case 2: $\delta_1 \leq \delta_2$.} Similarly as in Case~1, we obtain for $0<\abs{\xi} \leq \tfrac{1}{\delta_1\epsilon}$ that
	\[
	\absb{m_{\delta_1, \delta_2}(\xi)}=\abslr{\frac{\widehat{Q}_\rho(\delta_1 \xi)}{\widehat{Q}_\rho(\delta_2 \xi)}} \leq\textstyle \bigl(\min_{\overline{B_{\delta_2/(\delta_1\epsilon)}(0)}}\widehat{Q}_\rho\bigr)^{-1}\leq \bigl(\min_{\overline{B_{1/(\bar{\delta}\epsilon)}(0)}}\widehat{Q}_\rho\bigr)^{-1}, 
	\]
and for $\abs{\xi} \geq 1/(\delta_1\epsilon)$ by Lemma~\ref{lem:bms} that 
	\[
	\abs{m_{\delta_1, \delta_2}(\xi)} = \abslr{\frac{\widehat{Q}_\rho(\delta_1 \xi)}{\widehat{Q}_\rho(\delta_2 \xi)}} \leq C\left(\frac{\delta_2}{\delta_1}\right)^n \frac{\overline{\rho}\Bigl(\frac{1}{\delta_1\abs{\xi}}\Bigr)}{\overline{\rho}\Bigl(\frac{1}{\delta_2\abs{\xi}}\Bigr)} \leq C \left(\frac{\delta_2}{\delta_1}\right)^{1-\sigma} \leq C\bar{\delta}^{\sigma-1},
	\]
	with the second inequality due to \ref{itm:lower}. \smallskip
	
Finally, combining the two cases shows that $\absb{m_{\delta_1, \delta_2}(\xi)}\leq C=C(\rho, \bar{\delta})$ for all $\xi\neq 0$, which concludes the proof.
\end{proof}

Together with a density argument, the previous theorem shows that 
$H^{\rho,p,\d_1}(\R^n)=H^{\rho,p,\d_2}(\R^n)$ and $H^{\rho,p,\d_1}_0(\Omega) = H^{\rho,p,\d_2}_0(\Omega)$ for all $\d_1,\d_2 \in (0,1]$ or equivalently,
that 
\begin{align}\label{eq:spaces delta independent}
H^{\rho,p,\d}(\R^n)=H^{\rho,p}(\R^n) \quad \text{and} \quad H^{\rho,p,\d}_0(\Omega) = H^{\rho,p}_0(\Omega) \qquad \text{ for all $\d \in (0,1]$.}
\end{align}
In fact, by inspecting the proof of Theorem~\ref{th:comparison}, it is not hard to see that \eqref{eq:spaces delta independent} holds for all $\d >0$, which shows that our nonlocal function spaces do not depend on the horizon parameter $\d$.  Based on this observation, we obtain the following corollary as a consequence of~\eqref{estimateHsigmap}.
\color{black}
\begin{corollary}\label{cor:compactness}
There exists a constant $C=C(\rho,n,\Omega,p)>0$ such that
\[
\norm{u}_{H^{\sigma,p}(\R^n)} \leq C\norm{D_{\rho_\delta}u}_{L^p(\Omega_\d;\R^n)} \quad \text{for all $u \in H^{\rho,p,\d}_0(\Omega)$ and $\delta \in (0,1]$.}
\]
\end{corollary}

\begin{remark}
We observe that there is another way of proving Corollary~\ref{cor:compactness} that does not pass through the stronger statement of Theorem~\ref{th:comparison}. 
For this alternative argument, it suffices to require that the kernel $\rho$ satisfies \ref{itm:h0}-\ref{itm:derivatives} and
\[
\liminf_{\abs{x} \to 0} \rho(x)\abs{x}^{n+\sigma-1} > 0.
\] 
Indeed, one can compare $\rho_\d$ with the kernel from Example~\ref{ex:kernels}\,a) for $s=\sigma$ by arguing as in \cite[Theorem~7.2]{BMS24} and checking that the constants are independent of $\d$. When $p=2$, even \ref{itm:derivatives} is not necessary (cf.~\cite[Theorem~7.2]{BMS24}), so that also the truncated fractional gradients in~Remark~\ref{rem:relaxed} are covered.
\end{remark}

\color{black}
 In order to ensure the existence of convergent subsequences required for the forthcoming $\Gamma$-convergence result (see Theorem~\ref{th: Gamma convergence to 0}), we proceed with the following compactness statement.

\begin{lemma}[Convergent subsequences for vanishing horizon]\label{le:compactness} 
	Let $(\d_j)_j \subset (0,1]$ be a sequence with $\d_j \to 0$ and suppose that $u_j \in H^{\rho,p,\d_j}_0(\Omega)$ for each $j \in \N$ with
	\[
	\sup_{j \in \N} \norm{D_{\rho_{\d_j}} u_j}_{L^p(\Omega_{\d_j};\R^n)} < \infty.
	\]
	Then, there is a $u \in W^{1,p}_0(\Omega)$ (extended to $\Rn$ as zero) such that, up to a non-relabeled subsequence, 
	\[
	u_j \to u \ \  \text{in $L^p(\R^n)$} \quad \text{and} \quad D_{\rho_{\d_j}}u_j \weakto \nabla u \ \  \text{in $L^p(\R^n;\R^n)$ as $j \to \infty$.}
	\]
\end{lemma}
\color{black}
\begin{proof}
By Corollary~\ref{cor:compactness}, the sequence $(u_j)_j$ is bounded in $H^{\sigma,p}(\R^n)$. Since each $u_j$ is supported in $\overline{\Omega}$, we conclude from  the compact embedding $H^{\sigma,p}(\Rn)\hookrightarrow\hookrightarrow L^p_{\loc}(\Rn)$ (cf.~\ref{embeddingBessel}) that there is a $u \in L^p(\R^n)$ with $u = 0$ a.e.~in $\Omega^c$ such that, up to a non-relabeled subsequence, 
\begin{align*}
u_j \to u\qquad \text{ in $L^p(\R^n)$ as $j\to \infty$.}
\end{align*}   Moreover, up to extracting a potential further subsequence, we find that $D_{\rho_{\d_j}}u_j \weakto V$ in $L^p(\R^n;\R^n)$ for some $V\in L^p(\R^n;\R^n)$. To deduce that $V=\nabla u$, we compute for $\phi \in C_c^{\infty}(\R^n)$ that
	\begin{align*}
		\int_{\R^n} V\,\phi\,dx &= \lim_{j \to\infty}\int_{\R^n} D_{\rho_{\d_j}}u_j\,\phi\,dx\\
		&= - \lim_{j \to \infty}\int_{\R^n} u_j\,D_{\rho_{\d_j}}\phi\,dx\\
		&=-\int_{\R^n} u\,\nabla\phi\,dx,
	\end{align*}
	where the last equality follows from the localization result for the nonlocal gradients in Lemma~\ref{le:recovery}\,$(i)$. This shows that $u \in W^{1,p}_0(\Omega)$ (extended to $\R^n$ as zero) with $\nabla u=V$, which finishes the proof.
\end{proof}

\subsection[$\Gamma$-convergence $\d \to 0$]{$\Gamma$-convergence $\d \to 0$}\label{subsec:Gamma delta0}
We are now in the position to make the conjectured localization of our variational problems in the limit $\d\to 0$ rigorous, choosing $\Gamma$-convergence as a natural framework.

Before stating the theorem, let us collect the relevant objects. The family of vectorial energy functionals $(\Fcal_\d)_{\d\in (0,1]}$ with $\Fcal_\d :L^p(\R^n;\R^m) \to \R_\infty:=\R \cup \{\infty\}$ for $\d \in (0,1]$ is given by
\begin{equation}\label{eq:Fcald}
\Fcal_\d(u):=\begin{cases}
\displaystyle\int_{\Omega_\d} f(x,D_{\rho_\d} u)\,dx &\text{for $u \in H^{\rho,p,\d}_0(\Omega;\R^m)$},\\
\infty &\text{else,}
\end{cases}
\end{equation}
where $f:\Omega_1 \times \Rmn \to \R$ is a suitable Carath\'{e}odory integrand; considering that the functions in the domain of $\Fcal_\delta$ are defined on $\Rn$ with zero Dirichlet conditions in $\O^c$, we may take, without loss of generality, the integrals over the bounded set $\O_\d$. 

As prospective limit functional, we introduce 
 $\Fcal_0:L^p(\R^n;\R^m) \to \R_\infty$ given by
\begin{equation}\label{eq:Fcal0}
\Fcal_0(u):=\begin{cases}
\displaystyle\int_{\Omega} f(x,\nabla u)\,dx &\text{for $u \in W^{1,p}_0(\Omega;\R^m)$},\\
\infty &\text{else;}
\end{cases}
\end{equation}
here, functions in $W^{1,p}_0(\Omega;\R^m)$ are identified with their extension to $\R^n$ as zero.

\color{black}
\begin{theorem}[Localization for vanishing horizon via $\Gamma$-convergence] \label{th: Gamma convergence to 0}
Let $f:\Omega_1\times \Rmn \to \R$ be a Carath\'{e}odory integrand such that
\[
c\abs{A}^p-C\leq f(x,A) \leq C(1+\abs{A}^p) \quad \text{for a.e.~$x \in \Omega_1$ and all $A \in \Rmn$}
\]
with $c,C>0$. If $f(x,\cdot)$ is quasiconvex for a.e.~$x \in \Omega$, then the family $(\Fcal_\d)_{\d \in (0,1]}$ in \eqref{eq:Fcald} $\Gamma$-converges with respect to $L^p(\R^n;\R^m)$-convergence to the functional $\Fcal_0$ in \eqref{eq:Fcal0} as $\d \to 0$, that is,
\[
\Gamma(L^p)\text{-}\lim_{\d \to 0} \Fcal_\d =\Fcal_0.
\]
Additionally, $(\Fcal_\d)_\d$ is equi-coercive with respect to convergence in $L^p(\R^n;\R^m)$.
\end{theorem}

\begin{proof}
Let $(\d_j)_j \subset (0,1]$ be a sequence converging to 0 as $j \to \infty$. \smallskip

\textit{Equi-coercivity:} By the growth bound on $f$ from below and Corollary~\ref{cor:compactness}, we deduce that there are constants $c', C'>0$ such that
\[
\Fcal_{\d_j}(u)\geq c' \norm{u}_{H^{\sigma,p}(\R^n)} -C'
\]
for all $j \in \N$ and $u \in H^{\rho,p,\d_j}_0(\Omega;\R^m)$. This yields the equi-coercivity, given the compact embedding of $H^{\sigma,p}(\R^n)$ into $L_{\rm loc}^p(\R^n)$, cf.~\eqref{embeddingBessel}. \smallskip

\color{black}
\textit{Liminf-inequality:} Let $(u_j)_j \subset L^p(\R^n;\R^m)$ with $u_j \to u$ in $L^p(\R^n;\R^m)$. Assuming without loss of generality that
\[
\liminf_{j \to \infty} \Fcal_{\d_j} (u_j) = \lim_{j \to \infty} \Fcal_{\d_j}(u_j) < \infty,
\]
we have that $u_j \in H^{\rho,p,\d_j}_0(\Omega;\R^m)$ for each $j \in \N$ and 
\begin{align*}
\sup_{j \in \N}\norm{D_{\rho_{\d_j}}u_j}_{L^p(\Omega_{\d_j};\R^m)}< \infty,
\end{align*} 
due to the lower bound on $f$. 
Lemma~\ref{le:compactness} therefore yields a $u \in W^{1,p}_0(\Omega;\R^m)$ (extended to $\R^n$ as zero) such that $D_{\rho_{\d_j}}u_j \weakto \nabla u $ in $L^p(\R^n;\Rmn)$. If we use translation operators as in Lemma~\ref{le:translation} to define 
\begin{align*}
v_j:=\Qcal_{\rho_{\d_j}}u_j = Q_{\rho_{\d_j}}\ast u_j \quad \text{ for $j \in \N$,}
\end{align*} 
then $(v_j)_j$ is a bounded sequence in $W^{1,p}(\R^n;\R^m)$ with $\nabla v_j = D_{\rho_{\d_j}}u_j \weakto \nabla u$ in $L^p(\R^n;\Rmn)$ as $j \to \infty$. Consequently, it even holds that $v_j \weakto u$ in $W^{1,p}(\R^n;\R^m)$, considering that the functions $v_j$ are zero outside of $\Omega_1$ for each $j \in \N$.
\color{black} A standard lower semicontinuity result for functionals with quasiconvex integrands (cf.~\cite[Theorem~8.11]{Dac08}) then implies
\begin{align*}
\liminf_{j \to\infty} \Fcal_{\d_j}(u_j)&= \liminf_{j \to \infty} \int_{\Omega_{\d_j}} f(x,D_{\rho_{\d_j}}u_j)\,dx \\
&\geq \liminf_{j \to \infty}  \int_{\Omega} f(x,D_{\rho_{\d_j}}u_j)\,dx -C\abs{\Omega_{\d_j}\setminus \Omega}\\
&=\liminf_{j \to \infty} \int_{\Omega} f(x,\nabla v_j)\,dx\geq  \int_{\Omega}f(x,\nabla u)\,dx = \Fcal_0(u),
\end{align*}
which is the desired liminf-inequality.
\smallskip

\textit{Recovery sequence:} Without loss of generality, consider $u \in W^{1,p}_0(\Omega;\R^m)$. Then, we infer from Lemma~\ref{le:recovery}\,$(iii)$ that $u \in H^{\rho,p,\d_j}_0(\Omega;\R^m)$ for all $j \in \N$ with $D_{\rho_{\d_j}}u \to \nabla u$ in $L^p(\R^n;\Rmn)$ as $j\to \infty$. The upper bound on $f$ enables the application of Lebesgue's dominated convergence theorem to find
\[
\lim_{j \to \infty}\Fcal_{\d_j}(u) = \lim_{j \to \infty}\int_{\Omega_{\d_j}}f(x,D_{\rho_{\d_j}}u)\,dx = \int_{\Omega}f(x,\nabla u)\,dx = \Fcal_0(u).
\]
\color{black}
This shows that the constant sequence is a suitable recovery sequence.
\end{proof}

\begin{remark}
a) Under the assumptions of Theorem~\ref{th: Gamma convergence to 0}, for every $\d \in (0,1]$, the functional $\Fcal_\d$ admits a minimizer according to Theorem~\ref{th:existence}. \color{black} Therefore, standard properties of $\Gamma$-convergence imply that these minimizers converge, up to subsequence, to a minimizer of $\Fcal_0$ as $\d \to 0$. 

Referring to the literature, a closely related $\Gamma$-convergence result involving similar nonlocal gradients in the case $m=1$ can be found in \cite[Theorem~1.7]{MeS}. However, the latter does not feature the equi-coercivity required for the convergence of minimizers.
\smallskip

b) Since the definition of $D_{\rho_\d}u$ on $\O_\d$ only depends on the values of $u$ in $\O_{2\d}$, the Dirichlet condition in $\O^c$ can be equivalently replaced by prescribing zero values in $\O_{2\d}\setminus \O$. We remark that the papers \cite[Theorem~6.1]{BeCuMC23b} and \cite[Corollary 2]{CuKrSc23} on finite-horizon fractional gradients use a slightly different convention by considering the gradients on $\Omega$ and requiring Dirichlet conditions in $\O_\d \setminus \O_{-\d}$. Clearly, both settings are equivalent by a suitable renaming of the domain. The reason for our choice is that only the setting used here is meaningful for both limit passages $\delta \to 0$ and $\delta\to \infty$. 

\smallskip

\color{black} c) Theorem~\ref{th: Gamma convergence to 0} can readily be extended to non-zero complementary values,  that is, to admissible functions in the spaces $g+H^{\rho,p,\d}_0(\O;\R^m)$ for any given $g \in W^{1,p}(\R^n;\R^m)$. Indeed, since $g \in H^{\rho,p,\d}(\O;\R^m)$ for all $\d >0$ with 
\[
\norm{D_{\rho_\d}g}_{L^p(\R^n;\Rmn)} \leq \norm{Q_{\rho_\d}}_{L^1(\Rn)}\norm{\nabla g}_{L^p(\R^n;\Rmn)}=\norm{\nabla g}_{L^p(\R^n;\Rmn)},
\]
the argument follows through in the same manner with the domain of the $\Gamma$-limit being $g+W^{1,p}_0(\O;\R^m)$.
\end{remark}

\begin{example}
By applying Theorem~\ref{th: Gamma convergence to 0} to the kernels of Example~\ref{ex:kernels}, we obtain the localization of functionals as in \eqref{eq:Fcald} with nonlocal gradients associated to the following scaled kernels: \smallskip

a) For $\rho$ as in Example~\ref{ex:kernels}\,a), one finds that
\[
\rho_\d (x) = \d^{-n} \frac{w(x/\d)}{\abs{x/\d}^{n+s-1}} = \d^{s-1} \frac{w(x/\d)}{\abs{x}^{n+s-1}}, \quad x\in \R^n\setminus \{0\}.
\]
These scaled kernels coincide (up to a constant) with finite-horizon fractional gradients $D^s_\d$ from \cite{BeCuMC23,BeCuMC23b,CuKrSc23, KrS24}. Theorem~\ref{th: Gamma convergence to 0} then complements the localization result for $s \uparrow 1$ in \cite[Theorem~7]{CuKrSc23}. \smallskip

b) The scaled versions of the kernels $\rho$ in Example~\ref{ex:kernels}\,b) read as
\[
\rho_\d (x) = \d^{-n} \frac{w(x/\d)\log^\kappa(\d/\abs{x})}{\abs{x/\d}^{n+s-1}} = \d^{s-1} \frac{w(x/\d)(\log(\d)-\log(\abs{x}))^\kappa}{\abs{x}^{n+s-1}}, \quad x\in \R^n\setminus \{0\}.
\]

c) With $\rho$ as in Example~\ref{ex:kernels}\,c), the scaled kernels are given by
\[
\rho_\d (x) = \d^{-n} \frac{w(x/\d)}{\abs{x/\d}^{n+s(\abs{x/\d})-1}} = \d^{s(\abs{x/\d})-1} \frac{w(x/\d)}{\abs{x}^{n+s(\abs{x/\d})-1}}, \quad x\in \R^n\setminus \{0\}.
\]
\end{example}


\color{black}
\section{$\Gamma$-convergence $\d \to \infty$} \label{se: convergence in infinity}
We focus now on the asymptotics of the nonlocal gradient as the horizon $\d$ diverges to infinity.
As proven below, the associated limiting object is the Riesz fractional gradient. While this is to be expected for finite-horizon fractional gradients, surprisingly, the same holds when starting from any general nonlocal gradient within our setting. 
\color{black} This section is structured in parallel to Section~\ref{se: localization}, showing first the convergence of the nonlocal gradients to the Riesz fractional gradient, then providing a uniform compactness result, and finally, proving the $\Gamma$-convergence of the associated energy functionals. 

Let $\rho$ again be a non-negative radial kernel that satisfies \ref{itm:h0}-\ref{itm:upper} and \eqref{eq:normalized}, and assume throughout that $p \in (1,\infty)$ and $\O$ is a bounded Lipschitz domain. 
The lack of integrability of the fractional kernel on $\Rn$  calls for a different scaling procedure compared with Section~\ref{se: localization}.
\color{black} Precisely, for $\d \in (1/\epsilon,\infty)$ (so that $\overline{\rho}(1/\d)\not=0$), we now define 
\begin{align}\label{eq:kernelrhodelta}
\rho_\d(x):=\overline{\rho}\left(\frac{1}{\d}\right)^{-1}\rho\left(\frac{x}{\d}\right),
\end{align}
which corresponds to \eqref{eq:scaledprelim} with $c_\d:=\overline{\rho}(1/\d)^{-1}$. In this way, the values of $\rho_\delta$ on the unit sphere $\partial B_1(0)$ are normalized to $1$ for any $\delta$. In addition, we require that these kernels converge pointwise on $\R^n\setminus\{0\}$ as $\d \to \infty$, and set 
\begin{align}\label{eq:convergenceassumption}
\rho_{\infty}(x) := \lim_{\d \to \infty}\rho_\d(x)=\lim_{\d \to \infty}\overline{\rho}\left(\frac{1}{\d}\right)^{-1}\rho\left(\frac{x}{\d}\right), \quad x\in \R^n\setminus\{0\}.
\end{align}
With the scaling~\eqref{eq:kernelrhodelta}, the kernel function $Q_{\rho_\d}$ and its Fourier transform satisfy
\begin{equation}\label{eq:qrhohatinfty}
Q_{\rho_\d} = \overline{\rho}\left(\frac{1}{\d}\right)^{-1}Q_\rho\left(\frac{\,\cdot\,}{\d}\right) \quad \text{and} \quad \widehat{Q}_{\rho_\d}= \d^n\overline{\rho}\left(\frac{1}{\d}\right)^{-1} \widehat{Q}_{\rho}(\d \, \cdot\,). 
\end{equation}

Let us point out that our chosen scaling is, up to a constant, the only relevant one. Indeed, if there is a sequence $(\bar{c}_\d)_\d$ of positive reals such that $(\bar{c}_\d\rho(\cdot/\d))_\d$ converges pointwise as $\d \to \infty$, then for $x\in\partial B_1(0)$,
\[
\lim_{\d \to \infty}\bar{c}_\d \rho\left(\frac{x}{\d}\right) = \lim_{\d \to \infty} \bar{c}_\d\overline{\rho}\left(\frac{1}{\d}\right) =: \bar{c}.
\]
Therefore, we obtain for all $x \in \R^n \setminus\{0\}$ that
\[
\lim_{\d \to \infty}\bar{c}_\d \rho\left(\frac{x}{\d}\right)=\lim_{\d \to \infty} \bar{c}_\d\overline{\rho}\left(\frac{1}{\d}\right)\overline{\rho}\left(\frac{1}{\d}\right)^{-1}\rho\left(\frac{x}{\d}\right)=\bar{c}\rho_\infty(x).
\]

\color{black}

\color{black}
\subsection[Convergence of nonlocal gradients as $\d \to \infty$]{Convergence of nonlocal gradients as $\d \to \infty$}

This section is about establishing that the scaled nonlocal gradients converge to the Riesz fractional gradient as $\d \to \infty$. We commence with some bounds on $\rho_\d$ from~\eqref{eq:kernelrhodelta} and the limit kernel $\rho_{\infty}$ that will be used repeatedly later. Recall that $\eps>0$ is as in \ref{itm:h0}-\ref{itm:upper}, and that $\sigma, \gamma$ with $0<\sigma\leq \gamma<1$ are the parameters appearing in the hypotheses \ref{itm:lower} and \ref{itm:upper}, respectively. 
\color{black}
\begin{lemma}\label{le:wedge}
There exist constants $C,c>0$ such that for every $\d > 1/\epsilon$ and all $x\in B_{\d\epsilon}(0)\setminus\{0\}$,
\begin{equation}\label{eq:rhodwedge}
c\min\left\{\frac{1}{\abs{x}^{n+\sigma-1}},\frac{1}{\abs{x}^{n+\gamma-1}}\right\} \leq \rho_\d(x) \leq C\max\left\{\frac{1}{\abs{x}^{n+\sigma-1}},\frac{1}{\abs{x}^{n+\gamma-1}}\right\}. 
\end{equation}

In particular, it holds for all $x\in \R^n \setminus\{0\}$ that
\begin{align*} 
c\min\left\{\frac{1}{\abs{x}^{n+\sigma-1}},\frac{1}{\abs{x}^{n+\gamma-1}}\right\} \leq \rho_\infty(x) \leq C\max\left\{\frac{1}{\abs{x}^{n+\sigma-1}},\frac{1}{\abs{x}^{n+\gamma-1}}\right\}.
\end{align*}
\end{lemma}

\begin{proof}
Observe first that by \ref{itm:lower} and \ref{itm:upper}, there are constants $c', C'>0$ such that 
\begin{equation}\label{eq:comparingrho}
c'\left(\frac{t}{r}\right)^{n+\gamma-1}\overline{\rho}(t) \leq \overline{\rho}(r)\leq C' \left(\frac{t}{r}\right)^{n+\sigma-1}\overline{\rho}(t)
\end{equation}
for all $t,r\in (0,\eps)$ with $r\geq t$. 

Let $x\in B_{\delta\epsilon}(0)$. If $\abs{x} \geq 1$, we can apply~\eqref{eq:comparingrho} with the choice $r=\abs{x}/\d$ and $t=1/\d$ to find
\[
\frac{c'}{\abs{x}^{n+\gamma-1}}\leq \overline{\rho}\left(\frac{1}{\d}\right)^{-1}\overline{\rho}\left(\frac{\abs{x}}{\d}\right)\leq \frac{C'}{\abs{x}^{n+\sigma-1}}.
\]
 As for the case $0<\abs{x} \leq 1$, we resort to \eqref{eq:comparingrho} as well, but take $r=1/\d$ and $t=\abs{x}/\d$ instead, which gives
\[
\frac{1}{C'\abs{x}^{n+\sigma-1}}\leq \overline{\rho}\left(\frac{1}{\d}\right)^{-1}\overline{\rho}\left(\frac{\abs{x}}{\d}\right)\leq \frac{1}{c'\abs{x}^{n+\gamma-1}}.
\]
We conclude that \eqref{eq:rhodwedge} holds for suitably chosen constants $c, C$.
\end{proof}
As the next lemma shows, $\rho_{\infty}$ must be a fractional kernel, no matter the specific choice of $\rho$. 
This finding is a key ingredient for proving  that only Riesz fractional gradients can be obtained as the limit of increasing horizon nonlocal gradients with the scaled sequence of kernels $\rho_\d$.

\color{black}
\begin{lemma}[Limit kernel $\rho_\infty$ is fractional] \label{le: fractional kernel as limit}
	There is $s_\infty \in [\sigma,\gamma]$ such that $\rho_\infty$ in~\eqref{eq:convergenceassumption} satisfies
	\[
	\rho_\infty(x)=\rho^{s_\infty}(x):= \frac{1}{\abs{x}^{n+s_\infty-1}} \quad \text{for all $x \in \R^n\setminus\{0\}$.}
	\]
\end{lemma}
\begin{proof}
Our argument relies on proving that $\overline{\rho}_{\infty}$ is a multiplicative function. To this aim, we consider $r,t >0$ and compute that
	\begin{align*}
		\overline{\rho}_\infty(r\cdot t) &= \lim_{\d\to\infty}\overline{\rho}\left(\frac{1}{\d}\right)^{-1}\overline{\rho}\left(\frac{r \cdot t}{\d}\right) = \lim_{\d\to\infty}\overline{\rho}\left(\frac{1}{\d}\right)^{-1}\overline{\rho}\left(\frac{r}{\d}\right)\overline{\rho}\left(\frac{r}{\d}\right)^{-1}\overline{\rho}\left(\frac{r \cdot t}{\d}\right)\\
		&=\overline{\rho}_\infty(r) \lim_{\d\to\infty}\overline{\rho}\left(\frac{1}{\d/r}\right)^{-1}\overline{\rho}\left(\frac{ t}{\d/r}\right) = \overline{\rho}_\infty(r)\overline{\rho}_\infty(t).
	\end{align*}
	Since $\overline{\rho}_\infty$ is also locally bounded away from 0 by Lemma~\ref{le:wedge}, we deduce that $\overline{\rho}_\infty$ must be a power function (cf.~\cite[Chapter 3, Proposition 6]{AcD89}). 
	Together with Lemma~\ref{le:wedge}, it follows therefore that $\rho_\infty(x)=1/\abs{x}^{n+s_\infty-1}$ for some $s_\infty \in [\sigma,\gamma]$.
\end{proof}
\begin{remark}\label{rem:rhobars}
The parameter $s_\infty$ associated to a limit kernel $\rho_\infty$ can be determined directly from $\rho$ 
 via the limit
		\[
		s_\infty=\log(\overline{\rho}_\infty(1/e))-n+1=\lim_{\d \to \infty} \log\left(\overline{\rho}(1/\d)^{-1}\overline{\rho}\left(1/(e\d)\right)\right)-n+1.
		\] 
\end{remark}

\begin{example}\label{ex:sbar}
One observes that the kernels $\rho$ from Example~\ref{ex:kernels} satisfy \eqref{eq:convergenceassumption}, i.e., their rescaled versions converge pointwise. We identify the limiting fractional exponent $s_\infty$ for illustration:\smallskip

	\color{black} a) Let $\rho$ be as in Example~\ref{ex:kernels}\,a). Then, for $x\in \R^n\setminus\{0\}$, 
	\[
	\rho_\infty(x)=\lim_{\d \to \infty}\frac{\overline{w}(1/\d)^{-1}}{\d^{n+s-1}}\frac{w(x/\d)}{\abs{x/\d}^{n+s-1}}=\frac{1}{\abs{x}^{n+s-1}} =\rho^s(x),
	\]
which yields $s_\infty = s$. \smallskip 
	
	b) For the kernel $\rho$ of Example~\ref{ex:kernels}\,b), one obtains the same limit  as in a), that is, $\rho_\infty=\rho^s
	$, and hence, $s_\infty=s$. The detailed calculation reads
	\begin{align*}
		\rho_\infty(x)&=\lim_{\d \to \infty}\overline{w}(1/\d)^{-1}\frac{1}{\log^\kappa(\d)\d^{n+s-1}}\frac{w(x/\d)\log^\kappa(\d/\abs{x})}{\abs{x/\d}^{n+s-1}}=\frac{1}{\abs{x}^{n+s-1}}\lim_{\d \to \infty}\frac{\log^\kappa(\d/\abs{x})}{\log^\kappa(\d)}\\
		&=\frac{1}{\abs{x}^{n+s-1}}\lim_{\d \to \infty} \left(\frac{\log(\d)-\log(\abs{x})}{\log(\d)}\right)^\kappa=\frac{1}{\abs{x}^{n+s-1}}
	\end{align*}
for $x\in \R^n\setminus\{0\}$.\smallskip
	
	\color{black}
	c) In the case of $\rho$ from Example~\ref{ex:kernels}\,c), the limit fractional exponent becomes $s_\infty=s(0)$, as 
	\begin{align*}
		\rho_\infty(x)&=\lim_{\d \to \infty}\overline{w}(1/\d)^{-1}\frac{1}{\d^{n+s(1/\d)-1}}\frac{w(x/\d)}{\abs{x/\d}^{n+s(\abs{x}/\d)-1}}\\
		&=\lim_{\d \to \infty}\frac{1}{\d^{s(1/\d)-s(\abs{x}/\d)}\abs{x}^{n+s(\abs{x}/\d)-1}} =\frac{1}{\abs{x}^{n+s(0)-1}}
	\end{align*}
	for $x\in \R^n\setminus\{0\}$ shows. 
\end{example}

\color{black}
In the next step, we show that the nonlocal gradients converge to the fractional gradient induced by $\rho_\infty$ as $\d \to \infty$, see Proposition~\ref{prop:convergenceinfinite}.  The proof involves the following auxiliary result, which allows to control the integrability of the kernels during the limit passage.

\begin{lemma}\label{le:l1convergence}
It holds that 
\[
\rho_\d \,\min\{1,|\cdot|^{-1}\} \to \rho_{\infty}\, \min\{1,|\cdot|^{-1}\} \quad \text{in $L^1(\R^n)$ as $\d \to \infty$.}
\]
\end{lemma}
\begin{proof}
We already know the pointwise a.e.~convergence by \eqref{eq:convergenceassumption}, and it follows from \eqref{eq:rhodwedge} that
\begin{align*}\label{est45}
\mathbbm{1}_{B_{\d\epsilon}(0)} \rho_\d \min\{1,|\cdot|^{-1}\} \leq C\min\{1,|\cdot|^{-1}\}\max\left\{\frac{1}{\abs{\,\cdot\,}^{n+\sigma-1}},\frac{1}{\abs{\,\cdot\,}^{n+\gamma-1}}\right\}.
\end{align*}
Since the right-hand side is integrable, Lebesgue's dominated convergence theorem implies
\[
\mathbbm{1}_{B_{\d\epsilon}(0)}\rho_\d \, \min\{1,|\cdot|^{-1}\} \to \rho_{\infty}\, \min\{1,|\cdot|^{-1}\} \quad \text{in $L^1(\R^n)$ as $\d \to \infty$.}
\]

It remains to show that
\begin{equation}\label{eq:remainder}
\mathbbm{1}_{B_{\d\epsilon}(0)^c}\rho_\d \min\{1,|\cdot|^{-1}\} \to 0 \quad \text{in $L^1(\R^n)$ as $\d \to \infty$.}
\end{equation}
Considering that, in light of \ref{itm:h1}, $\rho(\cdot/\d) \leq C$ in $B_{\d\epsilon}(0)^c$ for some $C>0$, and $\rho(\cdot/\d)=0$ on $B_\delta(0)^c$ by~\eqref{eq:normalized}, we find 
\[
\mathbbm{1}_{B_{\d\epsilon}(0)^c}\rho_\d \leq C\mathbbm{1}_{B_\d(0)\setminus B_{\d\epsilon}(0)}\overline{\rho}(1/\d)^{-1}  \leq C\mathbbm{1}_{B_\d(0)\setminus B_{1}(0)}\overline{\rho}(1/\d)^{-1}, 
\]
given that $\d \epsilon >1$. This yields
\begin{align}
\int_{B_{\d\epsilon/2}(0)^c}\rho_\d(x) \min\{1,\abs{x}^{-1}\}\,dx &\leq C\overline{\rho}\left(\frac{1}{\d}\right)^{-1} \int_{B_\d(0)\setminus B_1(0)} \abs{x}^{-1}\,dx \nonumber \\ &\leq \begin{cases}
C\overline{\rho}(1/\d)^{-1}(\d^{n-1}-1) &\text{if $n>1$}, \\
C\overline{\rho}(1/\d)^{-1}\log(\d) &\text{if $n=1$.}\label{rs}
\end{cases}
\end{align}
In either case, the expression in~\eqref{rs} converges to 0 as $\d \to \infty$ in view of \ref{itm:lower}, which gives rise to  \eqref{eq:remainder} and finishes the proof.
\end{proof}
As a consequence, we now obtain the convergence of the nonlocal gradients to the Riesz fractional gradient as $\d \to \infty$ in the case of Sobolev functions.

\begin{proposition}[Convergence to fractional gradient as $\d \to \infty$]\label{prop:convergenceinfinite} 
For any $u \in W^{1,p}(\R^n)$ it holds that
\[
D_{\rho_\d}u \to D_{\rho_{\infty}}u=D^{s_\infty}u \quad \text{in $L^p(\R^n;\R^n)$ as $\d \to \infty$.}
\]
\end{proposition}
\begin{proof}
In light of \cite[Proposition~1]{EGM22} (cf.~also the proof of \cite[Proposition~3.5]{BMS24}), we deduce the estimate
\[
\norm{D_{\rho_\d}u-D_{\rho_{\infty}}u}_{L^p(\Rn;\Rn)}
\leq C \norm{u}_{W^{1,p}(\Rn)}\normb{(\rho_\d-\rho_\infty)\min\{1,|\cdot|^{-1}\}}_{L^1(\Rn)},
\]
for all $u\in W^{1,p}(\R^n)$, and the statement follows via Lemma~\ref{le:l1convergence}.
\end{proof}

\subsection[Compactness uniformly in $\delta \in (1/\epsilon,\infty)$]{Compactness uniformly in $\delta \in (1/\epsilon,\infty)$}

Next, we address the issue of compactness with the goal of deriving a counterpart of Lemma~\ref{le:compactness} in the setting of diverging horizon. This relies on the following analogue of the Poincar\'e-type inequality in Corollary~\ref{cor:compactness}. The proof is based on the comparison of the scaled nonlocal gradients  $D_{\rho_\d}$ with a suitable finite-horizon fractional gradient. 

\begin{proposition}\label{prop:poincare2}
There exists a constant $C=C(\rho,n,\Omega,p)>0$ such that
	\[
	\norm{u}_{H^{\sigma,p}(\R^n)} \leq C\norm{D_{\rho_\delta}u}_{L^p(\Omega_\d;\R^n)} \quad \text{for all $u \in H^{\rho,p,\d}_0(\Omega)$ and $\delta \in (1/\epsilon,\infty)$.}
	\]
\end{proposition}

\begin{proof}
Consider $D^\sigma_1:=D_{\rho^{\sigma}_1}$ induced by the kernel function
\begin{align}\label{eq:rho1sigma}
	\rho_1^\sigma= \frac{w}{| \cdot |^{n+\sigma-1}},
	\end{align}	
	where $w\in C^\infty_c(\R^n)$ is a non-negative radially decreasing function with $w(0)>0$ and $\supp w=B_1(0)$; note that $\rho_1^\sigma$ falls into the setting of Example~\ref{ex:kernels}\,a)  with $s=\sigma$ and recall that $\sigma$ is the parameter appearing in the hypothesis (H3) on $\rho$. 
	
 Then, by \eqref{estimateHsigmap} there is a constant $C>0$ such that
	\[
	\norm{u}_{H^{\sigma,p}(\R^n)} \leq C\norm{D^\sigma_1 u}_{L^p(\Omega_1;\R^n)} \quad \text{for all $u \in H^{\rho^\sigma_1,p}_0(\Omega)$}.
	\]
	\color{black}
The remaining proof shows that there is a constant $C>0$ independent of $\d$ such that
	\begin{align}\label{eq:estimates_graddelta}
	\norm{D^\sigma_1 \varphi}_{L^p(\R^n;\R^n)} \leq C\norm{D_{\rho_\d}\varphi}_{L^p(\R^n;\R^n)} \quad \text{for all $\varphi \in C_c^{\infty}(\R^n)$ and $\d \in (1/\epsilon,\infty)$,}
	\end{align}
	from which the claim follows after a density argument. 
	
	Let us define $m_\delta:\R^n\setminus \{0\}\to \R$ as
	\[
	m_\delta(\xi) := \frac{\widehat{Q}_{\rho^\sigma_1}(\xi)}{\widehat{Q}_{\rho_\d}(\xi)},
	\]
	and observe that, in light of~\eqref{eq:drhofourier}, 
	\[
	\widehat{D^\sigma_1 \varphi}= m_\delta \widehat{D_{\rho_{\delta}}\varphi} \qquad \text{for $\varphi\in C_c^\infty(\R^n)$}.
	\]
The estimate~\eqref{eq:estimates_graddelta} follows then directly from Fourier multiplier theory, once $m_\delta$ is confirmed to satisfy the Mihlin-H\"ormander condition with uniform constants, that is, 
	\begin{equation}\label{eq:mihlin2}
		\abs{\partial^{\alpha} m_\delta(\xi)} \leq C \abs{\xi}^{-\abs{\alpha}} 
	\end{equation}
		for all $\alpha \in \N_0^n$ with $\abs{\alpha} \leq n/2+1$ 
and $C>0$ a constant independent of $\d$. 
 
 \color{black}
  To this aim, observe that~\eqref{eq:bms2} implies for $\xi \not =0$ that
	\begin{align*}
		\abslr{\partial^{\alpha} \widehat{Q}_{\rho_\d}(\xi)} &=\abslr{\overline{\rho}(1/\d)^{-1}\d^{n+\abs{\alpha}}\partial^{\alpha}\widehat{Q}_\rho(\d\xi)}\\
		&\leq C \overline{\rho}(1/\d)^{-1}\d^{n+\abs{\alpha}}\abs{\d\xi}^{-\abs{\alpha}}\abslr{\widehat{Q}_\rho(\d\xi)}
		=C\abs{\xi}^{-\abs{\alpha}}\abslr{\widehat{Q}_{\rho_\d}(\xi)},
	\end{align*}
	with $C>0$ independent of $\delta$. 
	Since the same holds for $\widehat{Q}_{\rho^\sigma_1}$, we deduce via the Leibniz  and quotient rule that
	\[
	\abs{\partial^{\alpha} m_\d(\xi)} \leq C \abs{\xi}^{-\abs{\alpha}} \abs{m_\d(\xi)} \quad \text{for all $\xi \not =0$}.
	\]	
	
	Therefore, it remains to verify \eqref{eq:mihlin2} for $\alpha =0$, which corresponds to showing that $m_\delta$ is bounded independent of $\delta$. For simpler notation, we write $\langle \xi \rangle:=\sqrt{1+|\xi|^2}$ for $\xi\in \R^n$.  Since the estimate~\eqref{eq:bms1} in Lemma~\ref{lem:bms} along with~\eqref{eq:rho1sigma} allows us to deduce
	\[
	\widehat{Q}_{\rho^{\sigma}_1}(\xi) \leq C\abss{\xi}^{\sigma-1} \quad \text{for all $\xi \in \R^n$,}
	\]
	the proof of~\eqref{eq:mihlin2} for $\alpha=0$ can be reduced to verifying that
	\begin{align}\label{eq:estFourierQrho}
	\widehat{Q}_{\rho_\d}(\xi) \geq C\abss{\xi}^{\sigma-1} \quad \text{for all $\xi \in \R^n$.}
	\end{align}
	
	Let us first consider $\abs{\xi} \leq 1/(\d \epsilon)$. Then, in view of~\eqref{eq:qrhohatinfty}, 
	\begin{align}\label{eq:estQrhodelta11}
	\widehat{Q}_{\rho_\d}(\xi) \geq \overline{\rho}(1/\d)^{-1}\d^n \textstyle \min_{\overline{B_{1/\epsilon}(0)}} \widehat{Q}_{\rho} \geq C\geq C\abss{\xi}^{\sigma-1},
	\end{align}
	where $C$ is independent of $\d$ because $\overline{\rho}(1/\d)^{-1}\d^n \to \infty$ as $\d \to \infty$ by \ref{itm:upper}. For the case $\abs{\xi} > 1/(\d \epsilon)$, we use Lemma~\ref{lem:bms} and Lemma~\ref{le:wedge} to infer
	\begin{align}\label{eq:estQrhodelta12}
	\begin{split}
		\widehat{Q}_{\rho_\d}(\xi) &\geq C\overline{\rho}(1/\d)^{-1}\d^n\abs{\d\xi}^{-n}\overline{\rho}(1/\abs{\d\xi})\\
		&\geq C \abs{\xi}^{-n}\min\bigl\{\abs{\xi}^{n+\sigma-1},\abs{\xi}^{n+\gamma-1}\bigr\}\\
		&=C \min\bigl\{\abs{\xi}^{\sigma-1},\abs{\xi}^{\gamma-1}\bigr\}\geq C\abss{\xi}^{\sigma-1}.
		\end{split}
	\end{align}
	Finally,~\eqref{eq:estQrhodelta11} together with~\eqref{eq:estQrhodelta12} gives \eqref{eq:estFourierQrho}, and thus,~\eqref{eq:mihlin2}. This finishes the proof in light of the Mihlin-H\"ormander theorem (see e.g.~\cite[Theorem~6.2.7]{Gra14a}).
\end{proof}

\begin{remark} 
While~the previous proof is built on \eqref{eq:estimates_graddelta}, we mention that a statement parallel to Theorem~\ref{th:comparison} cannot be expected to hold for an unbounded parameter range of $\delta$. Indeed, this is due to the fact that the singular behavior of $\rho_{\infty}$ and $\rho_\d$ at the origin may be different, as one can see, for instance, from the two kernels in~Example~\ref{ex:sbar}\,b); they feature a stronger and weaker singularity than $\rho_\infty$, respectively.
\end{remark}

\color{black}
By combining Proposition~\ref{prop:poincare2} and Proposition~\ref{prop:convergenceinfinite}, we can now deduce the following compactness statement.

\begin{lemma}[Convergent subsequences for diverging horizon]\label{le:compactness2}
Let $(\d_j)_j \subset (1/\epsilon,\infty)$ be a sequence with $\d_j \to \infty$ and suppose that $u_j \in H^{\rho,p,\d_j}_0(\Omega)$ for each $j \in \N$ with
\[
\sup_{j \in \N} \norm{D_{\rho_{\d_j}} u_j}_{L^p(\Omega_{\d_j};\R^n)} < \infty.
\]
Then, there is a $u \in H^{s_\infty,p}_0(\Omega)$, such that, up to a non-relabeled subsequence, 
\[
u_j \to u \quad \text{in $L^p(\R^n)$} \quad \text{and} \quad D_{\rho_{\d_j}}u_j \weakto D^{s_\infty} u \quad \text{in $L^p(\R^n;\R^n)$ as $j \to \infty$.}
\]
Moreover, for every $\eta>0$ it holds that
\[
D_{\rho_{\d_j}}u_j \to D^{s_\infty} u \quad \text{in $L^p((\Omega_\eta)^c;\R^n)$ as $j \to \infty$.}
\]
\end{lemma}

\begin{proof}
We infer from Proposition~\ref{prop:poincare2} that $(u_j)_j$ is a bounded sequence in $H^{\sigma,p}(\R^n)$. Then, the compact embedding $H^{\sigma,p}(\R^n)\hookrightarrow\hookrightarrow L_{\loc}^p(\R^n)$ (see~\eqref{embeddingBessel}) together with the fact that each $u_j$ is supported in $\overline{\Omega}$ yields the existence of a $u \in L^p(\R^n)$ with $u = 0$ a.e.~in $\Omega^c$ such that, up to a non-relabeled subsequence, 
\begin{align}\label{eq:convergenceuj}
u_j \to u \qquad\text{ in $L^p(\R^n)$ as $j\to \infty$. }
\end{align}
After selecting a potential further subsequence, we find that $D_{\rho_{\d_j}}u_j \weakto V$ in $L^p(\R^n;\R^n)$ for some $V\in L^p(\R^n;\R^n)$. One can compute that
\begin{align*}
\int_{\R^n} V\,\phi\,dx &= \lim_{j \to\infty}\int_{\R^n} D_{\rho_{\d_j}}u_j\,\phi\,dx\\
&= - \lim_{j \to \infty}\int_{\R^n} u_j\,\Div_{\rho_{\d_j}}\phi\,dx=-\int_{\R^n} u\,\Div^{s_\infty}\phi\,dx,
\end{align*}  
for any $\phi \in C_c^{\infty}(\R^n;\R^n)$, 
where the last equality employs Proposition~\ref{prop:convergenceinfinite} adapted to the nonlocal divergence. This allows us conclude that $V=D^{s_\infty}u$ and $u \in H^{s_\infty,p}_0(\Omega)$. 

To prove the second part of the statement, we exploit that the nonlocal gradients on $(\Omega_\eta)^c$ can be expressed as a convolution. Precisely, let us define
\[
d_\d(z):=-\frac{z\rho_\d(z)}{\abs{z}^2} \quad  \text{and} \quad  d_\infty(z):=-\frac{z\rho_\infty(z)}{\abs{z}^2} \quad \text{for $z \in \R^n\setminus \{0\}$.}
\] 
For $\phi \in C_c^{\infty}(\Omega)$, we can compute in view of the radiality of $\rho$ that for any $x \in (\Omega_\eta)^c$,
\begin{align}\label{eq:Drhod}
D_{\rho_\d}\phi(x) = \int_{\R^n} -\frac{\phi(y)}{\abs{x-y}}\frac{x-y}{\abs{x-y}}\rho_\d(x-y)\,dy = (\mathbbm{1}_{B_{\eta}(0)^c}d_\d) * \phi (x);
\end{align}
since $\mathbbm{1}_{B_{\eta}(0)^c}d_\d \in L^1(\R^n)$, the identity~\eqref{eq:Drhod} can be extended via density to all $u \in H^{\rho,p,\d}_0(\Omega)$. In the same way, there is an analogous representation when considering the kernels $\rho_\infty$, that is,
\begin{align*}
D^{s_\infty} u = D_{\rho_\infty} u = (\mathbbm{1}_{B_{\eta}(0)^c}d_\infty) \ast u \qquad \text{on $(\Omega_\eta)^c$} 
\end{align*}  
for $u\in H_0^{s_\infty, p}(\Omega) = H_0^{\rho_\infty, p}(\Omega)$.

Furthermore, we observe that Lemma~\ref{le:l1convergence} induces the convergence
\begin{align*}
\mathbbm{1}_{B_{\eta}(0)^c}d_\d \to \mathbbm{1}_{B_{\eta}(0)^c}d_\infty \qquad \text{ in $L^1(\R^n)$ as $\d \to \infty$.}
\end{align*}  This allows us to conclude by Young's convolution inequality
and~\eqref{eq:convergenceuj} that
\[
\norm{D_{\rho_{\d_j}}u_j - D^{s_\infty} u}_{L^p((\Omega_\eta)^c;\R^n)}=\norm{(\mathbbm{1}_{B_{\eta}(0)^c}d_\d)*u_j - (\mathbbm{1}_{B_{\eta}(0)^c}d_\infty)*u}_{L^p((\Omega_\eta)^c;\R^n)}  \to 0 \ \text{as $\d \to 0$}.
\]
\end{proof}

\subsection[$\Gamma$-convergence $\d \to \infty$]{$\Gamma$-convergence $\d \to \infty$}
Based on the technical foundations provided in the previous sections, we are now in the position to prove the $\Gamma$-convergence for diverging horizon. We consider for $\delta\in (1/\eps, \infty)$ the functionals $\Fcal_\d:L^p(\R^n;\R^m) \to \R_\infty:=\R \cup \{\infty\}$ given by
\begin{equation}\label{eq:Fcald2}
\Fcal_\d(u):=\begin{cases}
\displaystyle\int_{\Omega_\d} f(x,D_{\rho_\d} u)\,dx &\text{for $u \in H^{\rho,p,\d}_0(\Omega;\R^m)$},\\
\infty &\text{else,}
\end{cases}
\end{equation}
where $f:\R^n \times \Rmn \to \R$ is a suitable Carath\'{e}odory integrand and $\rho_\delta$ is the scaled version of the kernel $\rho$, cf \eqref{eq:kernelrhodelta}. As made precise in Theorem~\ref{th:infty} below, the limiting object for $\delta\to \infty$ is the functional $\Fcal_\infty:L^p(\R^n;\R^m) \to \R_\infty$, 
\begin{equation}\label{eq:Fcalinf}
\Fcal_\infty(u):=\Fcal^{s_\infty}(u):=\begin{cases}
\displaystyle\int_{\R^n} f(x,D^{s_\infty} u)\,dx &\text{for $u \in H^{s_\infty,p}_0(\Omega;\R^m)$},\\
\infty &\text{else;}
\end{cases}
\end{equation}
The fractional parameter $s_\infty$ is here related to the kernel $\rho$ via $\lim_{\delta\to \infty}\rho_\delta  = \rho^{s_\infty}$, see~\eqref{eq:convergenceassumption}, Lemma~\ref{le: fractional kernel as limit}, and also Remark~\ref{rem:rhobars}. 

\begin{theorem}[$\Gamma$-convergence for diverging horizon]\label{th:infty}
Let $f:\R^n\times \Rmn \to \R$ be a Carath\'{e}odory integrand such that
\[
c\abs{A}^p - a(x)\leq f(x,A) \leq a(x)+C\abs{A}^p \quad \text{for a.e.~$x \in \R^n$ and all $A \in \Rmn$}
\]
with $c,C>0$ and $a\in L^1(\R^n)$. If $f(x,\cdot)$ is quasiconvex for a.e.~$x \in \Omega$, then the family $(\Fcal_\d)_{\d\in (1/\eps, \infty)}$ in \eqref{eq:Fcald2} $\Gamma$-converges with respect to $L^p(\R^n;\R^m)$-convergence to the functional $\Fcal_\infty$ in \eqref{eq:Fcalinf} as $\d \to \infty$, that is,
\[
\Gamma(L^p)\text{-}\lim_{\d \to \infty} \Fcal_\d =\Fcal_\infty.
\]
Additionally, the sequence $(\Fcal_\d)_\d$ is equi-coercive with respect to convergence in $L^p(\R^n;\R^m)$.
\end{theorem}

\begin{proof}
Let $(\d_j)_j \subset (1/\epsilon,\infty)$ be a sequence converging to $\infty$ as $j \to \infty$. \smallskip

\textit{Equi-coercivity:} From Proposition~\ref{prop:poincare2} and the lower bound on $f$, we deduce that
\[
\Fcal_{\d_j}(u)\geq C \norm{u}_{H^{\sigma,p}(\R^n)}  - \norm{a}_{L^1(\R^n)}
\]
for all $j \in \N$ and $u \in H^{\rho,p,\d_j}_0(\Omega;\R^m)$. The embedding~\eqref{embeddingBessel} now immediately gives the stated equi-coercivity.\smallskip

\textit{Liminf-inequality:} Consider a sequence $(u_j)_j \subset L^p(\R^n;\R^m)$ with $u_j \to u$ in $L^p(\R^n;\R^m)$ satisfying, without loss of generality,
\[
\liminf_{j \to \infty} \Fcal_{\d_j} (u_j) = \lim_{j \to \infty} \Fcal_{\d_j}(u_j) < \infty.
\] 
Then, $u_j \in H^{\rho,p,\d_j}_0(\Omega;\R^m)$ for each $j \in \N$ and $\sup_{j \in \N}\norm{D_{\rho_{\d_j}}u_j}_{L^p(\Omega_{\d_j};\R^m)}< \infty$, so that Lemma~\ref{le:compactness2} yields that $u$ lies in $H^{s_\infty,p}_0(\Omega;\R^m)$ with
\begin{align}\label{eq:convDrhoj}
 D_{\rho_{\d_j}}u_j \weakto D^{s_\infty} u\qquad \text{ in $L^p(\R^n;\Rmn)$ as $j\to \infty$.}
 \end{align}

Similarly to the proof of Theorem~\ref{th: Gamma convergence to 0}, we perform a translation to the classical gradient setting in order to estimate the integral contribution over $\Omega$. With
$v_j:=\Qcal_{\rho_{\d_j}} u_j \in W^{1,p}(\R^n;\R^m)$ for $j \in \N$, we have that $\nabla v_j = D_{\rho_{\d_j}}u_j$ due to
Lemma~\ref{le:translation}. Moreover, there exists a $v \in W^{1,p}_{\loc}(\R^n;\R^m)$ with $\nabla v = D^{s_\infty}u$ by \cite[Proposition~3.1\,$(i)$]{KreisbeckSchonberger}.
We therefore obtain in view of~\eqref{eq:convDrhoj} that $\nabla v_j \weakto \nabla v$ in $L^p(\Omega;\Rmn)$, and (up to translation by constants) that $v_j \weakto v$ in $W^{1,p}(\Omega;\R^m)$. A standard lower semicontinuity result for quasiconvex integrands (cf.~\cite[Theorem~8.11]{Dac08}) then yields
\begin{align}
\begin{split}\label{eq:inner}
\liminf_{j \to \infty} \int_{\Omega} f(x,D_{\rho_{\d_j}}u_j)\,dx &=\liminf_{j \to \infty} \int_{\Omega} f(x,\nabla v_j)\,dx\\
&\geq \int_{\Omega}f(x,\nabla v)\,dx = \int_{\Omega}f(x,D^{s_\infty}u)\,dx.
\end{split}
\end{align}

Regarding the integral contributions over $\Omega^c$, observe that  for any $\eta>0$, 
\begin{align*}
D_{\rho_{\d_j}}u_j \to D^{s_\infty}u\qquad \text{ in $L^p((\Omega_\eta)^c;\R^n)$ as $j\to \infty$.  }
\end{align*}
Together with the upper and lower bound on $f$ and Lebesgue's dominated convergence theorem, we obtain
\begin{align*}
\liminf_{j \to \infty} \int_{\Omega_{\d_j}\setminus\Omega} f(x,D_{\rho_{\d_j}}u_j)\,dx &\geq \liminf_{j \to \infty} \int_{\Omega_{\d_j}\setminus\Omega_\eta} f(x,D_{\rho_{\d_j}}u_j)\,dx - \int_{\Omega_{\eta}\setminus \Omega} a(x)\, dx \\
&=\int_{\R^n\setminus\Omega_\eta} f(x,D^{s_\infty}u)\,dx - \int_{\Omega_{\eta}\setminus \Omega} a(x)\, dx \\
&\geq \int_{\R^n\setminus\Omega} f(x,D^{s_\infty}u)\,dx - \int_{\Omega_\eta\setminus\Omega} 2a(x) +C\abs{D^{s_\infty}u}^p\,dx.
\end{align*}
Letting $\eta \to 0$ under consideration of $\abs{\partial\Omega}=0$ results in
\begin{align}\label{eq:outer}
\liminf_{j \to \infty} \int_{\Omega_{\d_j}\setminus\Omega} f(x,D_{\rho_{\d_j}}u_j)\,dx \geq  \int_{\R^n\setminus \Omega}f(x,D^{s_\infty}u)\,dx.
\end{align}
The desired liminf-inequality follows from adding \eqref{eq:inner} and~\eqref{eq:outer}.
\smallskip

\color{black}
\textit{Recovery sequence:} It suffices to consider $u \in H^{s_\infty,p}_0(\Omega;\R^m)$. Let $(u_k)_k \subset C_c^{\infty}(\Omega;\R^m)$ be a sequence such that $u_k \to u$ in $H^{s_\infty,p}_0(\Omega;\R^m)$ as $k\to \infty$. From Proposition~\ref{prop:convergenceinfinite} and the second part of Lemma~\ref{le:compactness2}, we deduce that
\begin{align*}
D_{\rho_{\d_j}}u_k \to D^{s_\infty}u_k \quad \text{ in $L^p(\R^n;\Rmn)$ as $j \to \infty$ for all $k \in \N$. }
\end{align*}
Hence, the upper bound on $f$ and a twofold application of Lebesgue's dominated convergence theorem shows
\[
\lim_{k \to \infty} \lim_{j \to \infty} \int_{\Omega_{\d_j}} f(x,D_{\rho_{\d_j}} u_k)\,dx = \lim_{k \to \infty} \int_{\R^n} f(x,D^{s_\infty}u_k)\,dx = \int_{\R^n} f(x,D^{s_\infty}u)\,dx.
\]
Finally, a recovery sequence is obtained by extracting a suitable diagonal sequence in the sense of Attouch~\cite[Lemma~1.15, Proposition~1.16]{Att84}.
\end{proof}

\begin{remark}
a) We find as a consequence of the $\Gamma$-convergence and equi-coercivity proven in Theorem~\ref{th:infty} that the minimizers of the functionals $\Fcal_\d$ in~\eqref{eq:Fcald2}, whose existence is guaranteed by Theorem~\ref{th:existence}, converge (up to a subsequence) in $L^p$ to a minimizer of $\Fcal^{s_\infty}$. In particular, this result applies to all kernels from Example~\ref{ex:kernels} with their limiting fractional exponents $s_\infty$ computed in Example~\ref{ex:sbar}. \smallskip
\color{black}

b) Note that Theorem~\ref{th:infty} can be readily generalized to functionals defined on the spaces $g + H^{\rho,p,\d}_0(\O;\R^m)$ with a given complementary value $g \in W^{1,p}(\R^n;\R^m)$, considering that Proposition~\ref{prop:convergenceinfinite} holds for all Sobolev functions in $W^{1,p}(\R^n)$. 
\end{remark}

\color{black}
\subsection*{Acknowledgments}
The authors would like to thank Jos\'e Carlos Bellido and Carlos Mora-Corral for helpful discussions on the topic. JC has been supported by the Spanish Agencia Estatal de Investigaci\'on through the project  PID2021-124195NB-C32, and by the Madrid Government (Comunidad de Madrid, Spain) under the multiannual Agreement with UAM in the line for the Excellence of the University Research Staff in the context of the V PRICIT (Regional Programme of Research and Technological Innovation). CK and HS gratefully acknowledge partial support by the Bayerische Forschungsallianz, through the project BayIntAn\_KUEI\_2023\_16.

\bibliography{bibliography}{}
\bibliographystyle{abbrv}
\end{document}